\documentclass{amsart}
\usepackage{graphicx}
\usepackage{amssymb}
\usepackage{amsfonts}
\usepackage{hyperref}
\usepackage{mathrsfs}
\swapnumbers
\newtheorem{theorem}{Theorem}[section]

\newtheorem{proposition}[theorem]{Proposition}
\theoremstyle{definition}

\newtheorem{assumption}[theorem]{Assumption}
\newtheorem{remark}[theorem]{Remark}

\newtheorem*{acknowledgment}{Acknowledgment}

\numberwithin{equation}{section}
\theoremstyle{plain}

\numberwithin{equation}{section} 
\numberwithin{figure}{section} 
\theoremstyle{plain}
\theoremstyle{plain}
\theoremstyle{remark}
\newtheorem*{acknowledgement*}{Acknowledgement}
\theoremstyle{example}

\newcommand{\cB}{{\mathcal B}}
\newcommand{\cC}{{\mathcal C}}

\newcommand{\cE}{{\mathcal E}}
\newcommand{\cF}{{\mathcal F}}

\newcommand{\cH}{{\mathcal H}}
\newcommand{\cI}{{\mathcal I}}

\newcommand{\cK}{{\mathcal K}}
\newcommand{\cL}{{\mathcal L}}

\newcommand{\cQ}{{\mathcal Q}}

\newcommand{\cX}{{\mathcal X}}
\newcommand{\cY}{{\mathcal Y}}

\newcommand{\te}{{\theta}}

\newcommand{\Om}{{\Omega}}
\newcommand{\om}{{\omega}}
\newcommand{\ve}{{\varepsilon}}
\newcommand{\del}{{\delta}}
\newcommand{\Del}{{\Delta}}
\newcommand{\gam}{{\gamma}}
\newcommand{\Gam}{{\Gamma}}

\newcommand{\sig}{{\sigma}}
\newcommand{\al}{{\alpha}}
\newcommand{\be}{{\beta}}
\newcommand{\ka}{{\kappa}}
\newcommand{\la}{{\lambda}}

\newcommand{\bbC}{{\mathbb C}}
\newcommand{\bbE}{{\mathbb E}}

\newcommand{\bbN}{{\mathbb N}}

\newcommand{\bbR}{{\mathbb R}}
\newcommand{\bbT}{{\mathbb T}}
\newcommand{\bbZ}{{\mathbb Z}}
\newcommand{\bbI}{{\mathbb I}}

\begin{document}
\title[]{Limit theorems for some time dependent expanding dynamical systems} 
 \vskip 0.1cm
 \author{Yeor Hafouta \\
\vskip 0.1cm
Department of Mathematics\\
The Hebrew University and the Ohio State University\\}%
\email{yeor.hafouta@mail.huji.ac.il}%

\thanks{ }
\dedicatory{  }
 \date{\today}

\maketitle
\markboth{Y. Hafouta}{Limit theorems} 
\renewcommand{\theequation}{\arabic{section}.\arabic{equation}}
\pagenumbering{arabic}

\begin{abstract}\noindent
In this paper we will prove
various probabilistic limit theorems for some classes of distance expanding sequential dynamical systems (SDS). Our starting point here is  certain sequential complex Ruelle-Perron-Frobenius (RPF) theorems which were proved in \cite{book} and \cite{ASIP me} using contraction properties of a complex version of the projective
Hilbert metric developed in \cite{Rug}. We will start from the growth rate of the variances of the underlying partial sums. This is well understood in the random dynamics setup, when the maps are stationary, but not in the SDS setup, where various growth rates can occur.
Some of our results in this direction rely on certain type of stability in these RPF theorems, which is one of the novelties of this paper. Then we will provide general conditions for several classical limit theorems to hold true in the sequential setup. 
Some of our general results mostly have applications for composition of random non-stationary map, while the conditions of the other results hold
true for general type of SDS. In the latter setup, results such as the Berry-Esse\'en theorem and the local central limit theorem were not obtained so far even for independent but not identically distributed maps, which is a particular case of the setup considered in this paper.
\end{abstract}

\section{Introduction}\label{sec1}
Probabilistic limit theorems for dynamical systems and Markov chains is a well studied topic. 
One  way to derive such results is relying on some quasi-compactness (or spectral gap)
of an appropriate transfer or Markov operator, together with a suitable perturbation theorem 
(see \cite{Neg1}, \cite{Neg2}, \cite{GH} and \cite{HH}). This  quasi-compactness  can often be verified only via 
an appropriate Ruelle-Perron-Frobenius (RPF) theorem, which is the main key for thermodynamic formalism type constructions.
Probabilistic limit theorems for
random dynamical systems and Markov chains in random dynamical environments were also studied in literature (see, for instance, \cite{Kifer-1996}, \cite{Kifer-1998}, \cite{book}, \cite{Aimino} ,\cite{drag}, \cite{drag1} and references therein). In these circumstances, the probabilistic behaviour of the appropriate process is determined by compositions of random operators, and not of a single operator, so no spectral theory can be exploit, and instead, many of these results rely on an appropriate version  
of the RPF theorem for random operators. 
 Relying on certain contraction properties of random complex transfer and Markov operators, with respect to a complex version of the Hilbert protective metric develpoed in \cite{Rug} (see also \cite{Dub1} and \cite{Dub2}), we proved in \cite{book} an RPF theorem for random complex operators and presented the appropriate random complex thermodynamic formalism type constructions,
which was one of the main keys in the proof of  versions of the Berry-Esseen theorem and the local central limit 
theorem for certain processes in random dynamical environment.

In recent years (see, for instance, \cite{Arno}, \cite{Conze}, \cite{Hyd2}, \cite{Zemer} and \cite{Nicol} and references therein) there has been a growing interest in proving limit theorems for sequences $X_n=T_0^n\textbf{x}_0$ of random variables generated by an appropriate random variable $\textbf{x}_0$ and compositions $T_0^n=T_{n-1}\circ T_{n-2}\circ\cdots\circ T_0$  of different maps $T_0,T_1,T_2,...$. Except for the random dynamical system case, in which the maps $T_j=T_{\xi_j}$ are chosen at random according to a stationary process $\{\xi_j\}$, the results obtained so far are certain versions of the central limit theorem (CLT), without close to optimal (distributional) convergence rate and corresponding local CLT's.  In this paper we will prove several limit theorems for some classes of maps $T_j$, as described in the next paragraphs. We stress that some, but not all, of our general results are counterparts of the results obtained in Chapter 7 of \cite{book} for sequences of maps instead of random stationary maps (and for non-uniformly distance expanding maps). We think that even though there is some overlap with this Chapter 7, it is important to have all the limit theorems formulated together as ``corollaries" of the sequential Ruelle-Perron-Frobenius theorem.

Let $T_j$ be a sequence of distance expanding maps, and $f_j,u_j$ sequences of  H\"older continuous (or differentiable) functions (uniformly in $j$). Let $\cL_0^{(j)}$ be the transfer operator generated by $T_j$ and the function $e^{f_j}$, 
and let $\cL_z^{(j)}(g)=\cL_0^{(j)}(g\cdot e^{zu_j}),\,z\in\bbC$ be the perturbations of $\cL_0^{(j)}$ corresponding to $u_j$.  The starting point of this paper 
is a  complex RPF theorem  for the sequences of transfer operator $\cL_z^{(j)},\,j\in\bbZ$. Applying this theorem we obtain
several limit theorems (finer than the CLT) for sequences of random variables of the form $S_{0,n}u(\textbf{x}_0)=\sum_{j=0}^{n-1}u_j\circ T_0^j(\textbf{x}_0)$, where $\mathbf{x}_0$ is distributed according some special (Gibbs) probability measure. We will also study certain stability properties of the RPF triplets, which will yield that the variance of the above random variables grow linearly fast in $n$, when the $T_j,f_j$ and $u_j$  lie in some neighborhood of a single distance expanding map $T$ and functions $f$ and $u$, respectively.

Some of the conditions of our general results hold true for non-random sequences of maps. The other results will be mainly applicable for certain classes random maps which are not necessarily stationary, a situation which  was not considered so far.
Several difficulties arise beyond the random dynamical setup (which was considered in \cite{book}), when the operators do not have the form $\cL_z^{(j)}=\cL_z^{\te^j\om}$  for an appropriate measure preserving system $(\Om,\cF,P,\te)$. For instance, in the stationary case many  limit theorems follow from the existence of the limit
$
\Pi(z)=\int\Pi_\om(z)dP(\om)=\lim_{n\to\infty}\frac 1n\sum_{j=0}^{n-1}\Pi_{\te^j\om}(z),
$
 where $\Pi_\om(z)$ is a certain random pressure function.
 In our case we have a sequence of pressure functions $\Pi_j(z),\,j\in\bbZ$ and the limits 
\begin{eqnarray*}
&\Pi(z)=\lim_{n\to\infty}\frac 1n\sum_{j=0}^{n-1}\Pi_{j}(z)
\end{eqnarray*}
do not necessarily exist. When $T_{j}=T_{\xi_j}, f_j=f_{\xi_j}$ and $u_j=u_{\xi_j}$ are chosen at random according to several classes of non-stationary sequences $\{\xi_j\}$ with certain mixing properties, we will show that the limits $\Pi(z)$ exist and are analytic functions of $z$, and then use that in order to apply some of our limit theorems.


\begin{acknowledgment}
I would like to thank Prof. Yuri Kifer for suggesting me to apply complex cone methods in the setup of time dependent dynamical systems, and for several references on these systems, as well.
\end{acknowledgment}

\section{Sequential dynamical systems: preliminaries and main results}\label{SDSsec}
Our setup consists of 
a bounded metric space $(\cX,d)$, a family $\{\cE_j:\,j\in\bbZ\}$ of  compact subsets of $\cX$, and a family of maps
$T_j:\cE_j\to\cE_{j+1},\,j\in\bbZ$. For any $j\in\bbZ$ and $n\geq1$ put
\[
T_j^n=T_{j+n-1}\circ\cdots\circ T_{j+1}\circ T_j.
\]
We will assume that one of the following three assumptions holds true

\begin{assumption}\label{Ass maps}
There exist constants $\xi>0, \gam>1,\, L,n_0\in\bbN,$ and $D>0$ so that
for any $j\in\bbZ$:

(i) For any $x\in\cE_j$,
\begin{equation}
T_j^{n_0}B_j(x,\xi)=\cE_{j+n_0}
\end{equation}
where  $B_j(x,\xi):=\{w\in\cE_j: d(w,x)<\xi\}$.

(ii) For any $x,x'\in\cE_{j+1}$ so that $d(x,x')<\xi$ we can write
\[
T_j^{-1}\{x\}=\{y_1,...,y_k\}\,\,\text{ and }\,\,T_j^{-1}\{x'\}=\{y'_1,...,y'_k\}
\]
where $k\leq D$ and for each $i=1,2,3,...,k$,
\begin{equation}\label{PairDist}
d(y_i,y_i')\leq\gam^{-1}d(x,x').
\end{equation}

(iii) There are $x_{1,j},x_{2,j},...,x_{L_j,j},\,L_j\leq L$ in $\cE_j$
so that 
\[
\cE_j=\bigcup_{s=1}^{L_j}B_j(x_{s,j},\xi).
\]
\end{assumption}

\begin{assumption}\label{Ass maps 2}
The inverse image $T_j^{-1}\{x\}$ of any point $x\in\cE_{j+1}$ under $T_j$ is at most countable, and
there exists a constant $\gam>1$ so that
the inverse images $\{x_i\}$ and $\{x_i'\}$ of any two points $x,x'\in\cE_{j+1}$ under $T_j$  can be paired so that (\ref{PairDist}) holds true.
\end{assumption}

\begin{assumption}\label{Ass maps 3}
There exist two sided sequences $(L_j)$, $(\sig_j)$, $(q_j)$ and $(d_j)$ so that $(L_j)$ is bounded and for each $j$ we have
 $\sigma_j>1$, $q_j,d_j\in\bbN$, $q_j<d_j$ and for any $x,x'\in\cE_{j+1}$ we can write 
\[
T_j^{-1}\{x\}=\{x_1,...,x_{d_j}\}\,\,\text{ and }\,\,T_j^{-1}\{x'\}=\{x'_1,...,x'_{d_j}\}
\]
where for any $i=1,2,...,q_j$ we have
\[
d(x_i,x'_i)\leq L_j\rho_{j+1}(x,x')
\]
while for any $i=q_{j}+1,...,d_j$ we have
\[
d(x_i,x'_i)\leq \sig_j^{-1}\rho_{j+1}(x,x').
\]
\end{assumption}

The arguments in \cite{MSU} show that Assumption \ref{Ass maps} is satisfied when the maps $T_j$ are locally distances expanding, uniformly in $j$. The second assumption holds true, for instance, when each $T_j$ is a map on the unit interval with a countable number of monotonicity intervals, and the absolute value of the derivative of $T_j$ on each one of these intervals is bounded from below by $\gam$. 
The main example we have in mind is the case where $\cX_j=M$ are the same
 compact and connected Riemannian manifold and each  $T_j:M\to M$ is a diffeomorphism satisfying the conditions in
 \cite{Vara} and \cite{castro} (with constant which are uniform in $j$).  This means that each $T_j$ locally expands distance only on some open region.
These conditions are satisfied when all the maps $T_j$ lie in a  $C^1$ neighborhood of some map $T$ satisfying the above conditions. Assumption \ref{Ass maps 2} holds true for intervals maps which only expand distance only on some pieces of the unit interval.


Next, for each integer $j$, let $f_j,u_j:\cE_j\to\bbR$ be H\"older continuous functions with 
exponent $\al$ which does not depend on $j$. 
For each integer $j$ and a complex number $z$, let 
$\cL_z^{(j)}$ be the linear operator which maps complex valued
functions $g$ on $\cE_j$ to complex valued functions $\cL_z^{(j)}g$ on $\cE_{j+1}$ by the formula
\[
\cL_z^{(j)}g(x)=\sum_{y\in T_j^{-1}\{x\}}e^{f_j(y)+zu_j(y)}g(y).
\]
For each $j\in\bbZ$ and $n\in\bbN$ set 
\[
\cL_z^{j,n}=\cL_z^{(j+n-1)}\circ\cdots\circ\cL_z^{(j+1)}\circ\cL_z^{(j)}.
\]
Let $\cH_j$ be the (Banach) space of all H\"older continuous functions $g:\cE_j\to\bbC$ with exponent $\al$, equipped with the norm $\|g\|_\al=\|g\|_\infty+v(g)$, where $\|g\|_\infty=\sup|g|$ and 
\[
v_\al(g)=v_{\al,\xi}(g)=\sup\Big\{\frac{|g(x)-g(x')|}{d^\al(x,x')}:\,0<d(x,x')\leq\xi\Big\}.
\]
In the case when $\al=1$ and $\cX_j$ is a Riemannian  manifold
we will also consider the case when $f_j$  and $u_j$ are $C^1$ functions and in this case we also consider the ''variation" $\nu(g)=\sup\|Dg\|$, where $Dg$ is the differential of $g$.
We will denote here by $\cH_j^*$ the dual of the Banach space $\cH_j$.
Finally, for any family of functions $\{g_j:\cE_j\to\bbC:\,j\in\bbZ\}$ we set for each integer $j$ and $n\geq1$,
\[
S_{n,j}g=\sum_{k=j}^{j+n-1}g_k\circ T_j^k.
\]
We will work in this paper under the following 
\begin{assumption}\label{f u ass}
Under Assumptions \ref{Ass maps} and \ref{Ass maps 2},
the norms $\|f_j\|$
and $\|u_j\|$ are bounded in $j$ by some constant $B$. 
\end{assumption}

\begin{assumption}\label{Bound ass}
 Assumption \ref{Ass maps 3} holds true and there exist constants $B>0$ and $s\in(0,1)$ so that 
we have $\sup_j\sup |f_j|\leq B$, $\sup_j\sup_x\sum_{y\in T_j^{-1}\{x\}}e^{f_j(y)}\leq B$,
\[
\sup{f_j}-\inf{f_j}<\ve_j\,\,\text{ and }\,\,v(\phi_j)<\ve'_j
\]
where $\ve_j,\ve'_j$ are sequence of positive constants satisfying
\begin{equation}\label{s sup}
\sup_{j}e^{\ve_j}\frac{q_jL_j^\al-(d_j-q_j)\sig_j^{-\al}}{d_j}\leq s
\end{equation}
and $\sup_j\ve'_j\leq B$.
\end{assumption}

From now on, we will refer to the constants  appearing in 
 in Assumptions \ref{Ass maps}, \ref{Ass maps 2}, \ref{Ass maps 3}, \ref{f u ass} and \ref{Bound ass} as the  ``initial parameters".

Our starting point in this paper is the following 

\begin{theorem}\label{RPF SDS}
Suppose that one of Assumptions \ref{Ass maps}, \ref{Ass maps 2}  and \ref{Ass maps 3} hold true and that either Assumption \ref{f u ass} or Assumption \ref{Bound ass} are satisfied (depending on the case) hold true.
Then there exists a neighborhood $U$ of $0$, which depends only on the initial parameters, so that for any $z\in U$
there exist families
$\{\la_j(z):\,j\in\bbZ\}$, $\{h_j^{(z)}:\,j\in\bbZ\}$ and $\{\nu_j^{(z)}:\,j\in\bbZ\}$ 
consisting of a nonzero complex number 
$\la_j(z)$, a complex function $h_j^{(z)}\in\cH_j$ and a 
complex continuous linear functional $\nu_j^{(z)}\in\cH_j^*$ such that:

(i) For any $j\in\bbZ$,
\begin{equation}\label{RPF deter equations-General}
\cL_z^{(j)} h_j^{(z)}=\la_j(z)h_{j+1}^{(z)},\,\,
(\cL_z^{(j)})^*\nu_{j+1}^{(z)}=\la_j(z)\nu_{j}^{(z)}\text{ and }
\nu_j^{(z)}(h_j^{(z)})=\nu_j^{(z)}(\textbf{1})=1
\end{equation} 
where $\textbf{1}$ is the function which takes the constant value $1$.
When $z=t\in\bbR$ then 
 $\la_j(t)>a$ and the function $h_j(t)$ takes values at some interval $[c,d]$, where $a>0$ and $0<c<d<\infty$ depend only on the initial parameters. Moreover, $\nu_j^{(t)}$ is a probability measure which assigns positive mass to open subsets of $\cE_j$ and the equality 
$\nu_{j+1}(t)\big(\cL_t^{(j)} g)=\la_j(t)\nu_{j}^{(t)}(g)$ holds true for any 
bounded Borel function $g:\cE_j\to\bbC$.

(ii) The maps 
\[
\la_j(\cdot):U\to\bbC,\,\, h_j^{(\cdot)}:U\to \cH_j\,\,\text{ and }
\nu_j^{(\cdot)}:U\to \cH_j^*
\]
are analytic and there exists a constant $C>0$, which depends only on the initial parameters
such that 
\begin{equation}\label{UnifBound}
\max\big(\sup_{z\in U}|\la_j(z)|,\, 
\sup_{z\in U}\|h_j^{(z)}\|,\, \sup_{z\in U}
\|\nu^{(z)}_j\|\big)\leq C,
\end{equation}
where $\|\nu\|$ is the 
operator norm of a linear functional $\nu:\cH_j\to\bbC$.
Moreover, there exist a constant $c>0$, which depends only on the initial parameters, so that $|\la_j(z)|\geq c$ 
and $\min_{x\in \cE_j}|h_j^{(z)}(x)|\geq c$ for any integer $j$ and $z\in U$.

(iii) There exist constants $A>0$ and $\del\in(0,1)$, 
which depend only on the initial parameters,
 so that for any $j\in\bbZ$, $g\in\cH_j$
and $n\geq1$,
\begin{equation}\label{Exponential convergence}
\Big\|\frac{\cL_z^{j,n}g}{\la_{j,n}(z)}-\nu_j^{(z)}(g)h^{(z)}_{j+n}\Big\|\leq A\|g\|\del^n
\end{equation}
where $\la_{j,n}(z)=\la_{j}(z)\cdot\la_{j+1}(z)\cdots\la_{j+n-1}(z)$. Moreover, the probability measures $\mu_j,\,j\in\bbZ$ given by $d\mu_j=h_j^{(0)}d\nu_j^{(0)}$ 
satisfy that $(T_j)_*\mu_j=\mu_{j+1}$ and that
 for any $n\geq1$ and $f\in\cH_{j+n}$,
\begin{equation}\label{ExpDec}
\big|\mu_j(g\cdot f\circ T_j^n)-\mu_j(g)\mu_{j+n}(f)\big|\leq A\|g\|\mu_{j+n}(|f|)\del^n.
\end{equation}
\end{theorem}
Under Assumptions \ref{Ass maps} Theorem \ref{RPF SDS} was essentially proved in Chapter 5 of \cite{book}, and under Assumption \ref{Ass maps 3} it was proved in \cite{ASIP me}. The proof under Assumption \ref{Ass maps 2} proceeds similarly to the proof under Assumption \ref{Ass maps}, and for that reason the details are omitted.
 
We note that when  the $\{T_j\}$ are "sequentially non-singular" the measures $\mu_j$ are absolutely continuous, 
as stated in the following

\begin{proposition}\label{NonSinThm}
Let $\textbf{m}_j,\,j\in\bbZ$ be a family of probability measures on $\cE_j$, which assign positive mass to open sets, so that for each $j$
we have $(T_j)_*\textbf{m}_j\ll\textbf{m}_{j+1}$ and that $e^{-f_j}=\frac{d(T_j)_*\textbf{m}_j}{d\textbf{m}_{j+1}}$. Then for any $j$ we have
$\la_j(0)=1$ and $\nu_j^{(0)}=\textbf{m}_j$.
\end{proposition}
When $\cE_j=\cX$ we can always take $\textbf{m}_j=\textbf{m}$ for some fixed $\textbf{m}$ (e.g. a volume measure
when $\cX$ is a Riemannian manifold). Theorem \ref{NonSinThm} is proved exactly as in \cite{Annealed} (see Proposition 3.1.2 there).


Our first result in this paper is the following stability theorem:
\begin{theorem}\label{stability}
Let $r>0$ be so that $\bar B(0,2r)=:\{z\in\bbZ:\, |z|\leq 2r\}\subset U$ and set
$K=\bar B(0,r)$. Then 
for any $\ve>0$ there exists $\del>0$ with the following property: if $T_{1,j},\,j\in\bbZ$ is a family of maps and 
$f_{1,j},u_{1,j}\in\cH_j$, where $j\in\bbZ$, are families of  functions which satisfy one of Assumptions \ref{Ass maps}-\ref{Ass maps 3} and one of Assumptions \ref{f u ass} and \ref{Bound ass} (in accordance with $T_j$ and with the same initial parameters), and for any $z\in K$ and  $j\in\bbZ$,
\begin{equation}\label{ve 0}
\|\cL_z^{(j)}-\cL_{1,z}^{(j)}\|<\del
\end{equation}
where $\cL_{1,z}^{(j)}$ is the operator defined similarly to $\cL_{z}^{(j)}$ but with $T_{1,j},f_{1,j}$ and $u_{1,j}$ in place of 
$T_{j},f_{j}$ and $u_{j}$,
then for any integer $j$ and $z\in K$ we have 
\[
\max\big(|\la_j(z)-\la_{1,j}(z)|,\,\|h_{j}^{(z)}-h_{1,j}^{(z)}\|,
\,\|\nu_{j}^{(z)}-\nu_{1,j}^{(z)}\|\big)<\ve
\]
where $\{\la_{1,j}(z):\,j\in\bbZ\}, \{h_{1,j}^{(z)}:\,j\in\bbZ\}$ and $\{\nu_{1,j}^{(z)}:\,j\in\bbZ\}$,\, $z\in U$ are the RPF families corresponding to the operators $\cL_{1,z}^{(j)}$.
\end{theorem}
The relation between this theorem and to ``limit theorems" is that it is the key ingredient in the proof of Theorem \ref{Variance SDS} (ii) below.

\subsection{Probabilistic limit theorems: general formulations}
In this section we will formulate general theorems, while some of them will mainly be applicable to the case of non-stationary random environment, Theorem \ref{Variance SDS}, the Berry-Esseen theorem the exponential, the exponential concentration inequalities and the moderate deviations theorems stated below are applicable for non-random sequences of maps. We especially want to stress that Theorem \ref{Gen LLT} is a modification of the local central limit theorem from \cite{book} in the sequential setup, and its main importance in its applications to the non-stationary random environments studied in Section \ref{NonStat}. The main reason this theorem (and Theorem \ref{Gen: Var, LD, MD}) are formulated here is in order not to overload the exposition with the precise definitions of the different non-stationary environment we have in mind.

\subsubsection{The variance and the moments}
We begin with the variances:
\begin{theorem}\label{Variance SDS}
(i) The variances $\text{var}_{\mu_k}(S_{k,n}u),\,k\in\bbZ$ do not
converge to $\infty$ as $n\to\infty$ if and only if there exists a family of functions $Y_s:\cE_s\to\bbR,\,s\in\bbZ$
and a constant $C>0$
so that for any $s\in\bbZ$ we have
\[
\int Y_s^2(x)d\mu_s(x)<C\,\,\text{ and }\,\,u_s-\mu_s(u_s)=Y_{s+1}\circ T_s-Y_s,\,\mu_s-\text{a.s.}
\]
When viewed as a random variable, the function $Y_k$ is a member of the subspace 
of $L^2(\cE_k,\mu_k)$ generated by the functions $\{u_k\circ T_k^d-\mu_k(u_k): d\geq 0 \}$ where $T_k^0:=\text{Id}$.
Moreover, the function $Y_k$ can be chosen to be a member of $\cH_k$ so that $\|Y_k\|$ is bounded 
in $k$, and in this case the equalities $u_s(x)-\mu_s(u_s)=Y_{s+1}(T_sx)-Y_s(x)$ hold true for any $x\in\cE_s$. 

(ii) Suppose that $\cE_k=\cX$ for each $k$. Let $T:\cX\to\cX$ be a map so that one of Assumptions \ref{Ass maps}-\ref{Ass maps 3} are satisfies with $T_j=T$, and let $u,f:\cX\to\bbR$ be members of $\cH=\cH_j$, so that $\|f\|$ and $\|u\|$ do not exceed $B$, and
$u$ does not admit a co-boundary representation with respect to $T$. In the case when Assumption \ref{Ass maps 3} holds true we also require Assumption \ref{Bound ass} to hold with $T_j=T, f_j=f$ and $u_j=u$ (with the same initial parameters).
Then there exists $\ve_0>0$, which depends only on the initial parameters,  so that the following holds true: if 
\begin{equation}\label{ve }
\sup_{k\in\bbZ}\|\cL_{z}^{(k)}-\cL_z\|_\infty\leq \ve_0
\end{equation}
for any $z$ in some neighborhood of $0$, where $\cL_z$ is the transfer operator generated by $T$ and $f+zu$, then 
\[
\inf_{k}\text{var}_{\mu_k}(S_{k,n}u)\geq \del_0 n
\] 
for some $\del_0>0$ and all sufficiently large $n$.

\end{theorem}
Suppose that $\cX_j=M$ are all the same  Riemannian manifold, and that the maps $T_j$ satisfy the conditions from \cite{castro}. Assume also that  all of the $T_j$'s  lie a sufficiently small  $C^1$-ball of a single map which satisfies these conditions, and  that the functions $f_j$ and $u_j$ all lie in a sufficiently small  ball (in the $C^1$-norm) around $f$ and $u$, respectively.  Then (\ref{ve }) is satisfied with the norm $\|g\|=\sup|g|+\sup|Dg|$ (i.e. we take $\al=1$, see Proposition 5.3  in \cite{castro}). We note that the size of the neighborhoods of $T,f$ and $u$ which is needed depends only on the initial parameters and on the limit $\sig_u^2=\lim_{n\to\infty}\text{Var}_\mu(\sum_{j=0}^{n-1}u\circ T^j)>0$, where $\mu=\mu_{T,f}$ is the appropriate $T$-invariant Gibbs measure corresponding to the function $f$. In fact, it is possible to give a precise formula for the radius of the balls around $T,f$ and $u$.
 Another example is intervals maps with finite number of monotonicity intervals which do not depend on $j$, where on each one of them each $T_j$ and $T$ is either expanding or contracting. If each $T_j$ is obtained from $T$ by  perturbing each inverse branch of $T$ in some H\"older norm, and $f_j$ and $u_j$ are small perturbations of $f$ and $u$ in this norm,  then (\ref{ve 0}) will hold true in the appropriate H\"older norm. Similar examples can be given in rectangular regions in $\bbR^d$ for $d>1$.

Next, for any $n\geq1$ set 
\[
\sig_{0,n}=\sqrt{\text{var}_{\mu_0}(S_{0,n}u)}
\]
and let $\textbf{x}_0$ be a $\cE_0$-valued random variable which is distributed according to $\mu_0$.
Then, relying on  (\ref{Exponential convergence}) and the arguments in \cite{Conze}, it follows that
 $\frac{\bar S_{0,n}u(\textbf{x}_0)-\mu_0(S_{0,n}u)}{\sig_{0,n}}$ converges in distribution towards  the standard normal
 law, when $\sig_{0,n}^2$  converges to $\infty$ as $n\to\infty$ (this essentially means that the quadratic variations of the martingales constructed in the proof of Theorem \ref{Mart thm} converge as $n\to\infty$, after a proper normalization).
When the variances grow faster than $n^{\frac23}$ then we are able to prove a
self normalized Berry-Esseen theorem:

\subsubsection{Berry-Eseen\'en and a local CLT}
\begin{theorem}\label{Berry-Esseen}
Suppose that 
\[
\lim_{n\to\infty}{\sig_{0,n}}{n^{-\frac13}}=\infty
\]
and set $\bar S_{0,n}u=S_{0,n}u-\mu_0(S_{0,n}u)$.
Then there exists a constant $C>0$ so that for any $n\geq 1$ and $r\in\bbR$,
\begin{equation}
\Big|\mu_0\{x\in\cE_0: \bar S_{0,n}u(x)\leq r\sig_{0,n}\}-\frac1{\sqrt {2\pi}}\int_{-\infty}^{r}e^{-\frac12 t^2}dt\Big|\leq Cn\sig_{0,n}^{-3}.
\end{equation}
In particular, when $\sig_{0,n}^2$ grows linearly fast in $n$ the above left hand side does not exceed $C_1n^{-\frac12}$ for some constant $C_1$. 
\end{theorem}
Note that we obtain here optimal convergence rate (i.e. rate of order $n^{-\frac12}$) in the circumstances of Theorem \ref{Variance SDS} (ii).

For the sake of completeness, we will formulate the following theorem whose proof is carried out exactly as in \cite{HafEdge}:
\begin{theorem}\label{MomThm.0}

(i) By possibly decreasing $r$, where $r$ comes from Theorem \ref{stability}, we can define analytic functions $\Pi_j:B(0,r)\to\bbC,\,j\in\bbZ$
so that 
\[
\Pi_j(0)=0,\, \la_j(z)/\la_j(0)=e^{\Pi_j(z)}\, \text{ and }\, 
|\Pi_j(z)|\leq c_0
\]
for any $z\in B(0,r)$, for some constant $c_0$ which does not depend on $j$
and $z$.


(ii) There exists a constant $R_1$ so that for any integer $j$ and $n\geq1$,
\[
\big|\mu_j(S_{j,n}u)-\sum_{m=0}^{n-1}\Pi'_{j+m}(0)\big|\leq R_1
\]
where $\Pi'_{j+m}(0)$ is the derivative of $\Pi_{j+m}$ at $z=0$.

(iii) Suppose that $\mu_j(u_j)=0$ for any integer $j$.
For any $k\geq 2$, $s\in\bbZ$ and $n\geq1$, set
\[
\gam_{j,k,n}=n^{-[\frac k2]}\int_{\cE_j} (S_{j,n}u(x))^kd\mu_j(x)
\,\text{ and }\, \Pi_{j,k,n}=n^{-1}\sum_{m=0}^{n-1}\Pi_{j+m}^{(k)}(0)
\]
where $\Pi_{j+m}^{(k)}(0)$ is the $k$-th derivative of the function $\Pi_{j+m}$ at $z=0$.
Then there exist constants $R_k>0$, $k\geq2$,
so that for any even $k\geq 2$,
\begin{equation}\label{Even mom}
\max\Big(\big|\gam_{j,k,n}-C_k(\gam_{j,2,n})^{\frac k2}\big|,\,
\big|\gam_{j,k,n}-C_k(\Pi_{j,2,n})^{\frac k2}\big|\Big)
\leq \frac{R_k}n
\end{equation}
where $C_k=2^{-\frac k2}(\frac{k}{2}!)^{-1}k!$, while with $D_k=\frac{k!}{3!}2^{-\frac12(k-3)}(\frac{k-3}{2}!)^{-1}$
for any odd $k\geq 3$,
\[
\max\big(\big|\gam_{j,k,n}-D_k(\gam_{j,2,n})^{\frac {k-3}2}\gam_{j,3,n}\big|,\,\big|\gam_{j,k,n}-D_k(\Pi_{j,2,n})^{\frac {k-3}2}\Pi_{j,3,n}\big|\Big)\leq \frac{R_k}n.
\]
\end{theorem}
An immediate consequence of Theorem \ref{MomThm.0} is that $|\gam_{j,k,n}|$ is uniformly bounded in $j$ and $n$, for each $k$. Using the Markov inequality, together with (\ref{Even mom}) with $k=4$, we obtain that for any $\ve>0$ and $j\in\bbZ$,
\[
\mu_j\big\{x:\,|S_{j,n}u(x)-\mu_j(S_{j,n}u)|\geq\ve n\big\}\leq Cn^{-2}\ve^{-4}
\]
for some $C>0$. Therefore, by the Borel-Cantelli Lemma,
\[
\lim_{n\to\infty}\frac{S_{j,n}u(x)-\mu_j(S_{j,n} u)}{n}=0,\,\,\,\mu_j-\text{a.s.}
\]
namely, a sequential version of the strong law of large numbers holds true.

When $T_j,f_j,u_j$ are chosen at random according to some type of (not necessarily stationary) sequence of random variables, then, in Section \ref{NonStat}
we obtain  almost sure converges rate of the form
\[
\Big|\frac1n\mu_j(S_{j,n}u)-p\Big|\leq \frac{R_{1,\om}}{n}
\]
where $p$ is some constant and $R_{1,\om}$ is some random variable.
When $\mu_j(u_j)=0$ we also derive that for any $k\geq2$,
\begin{equation}\label{MomRateTemp}
\big|\gam_{j,k,n}-\gam_k\big|\leq R_{\om,k}n^{-\frac12}\ln n
\end{equation}
where $R_{\om,k}$ is some random variable and $\gam_k$ is some constant.
 Using (\ref{MomRateTemp}) with $k=2$, when $\sig^2=\gam_2>0$,  
and almost optimal convergence rate in the central limit theorem of the form
\begin{equation}\label{BE temp}
\sup_{s\in\bbR}\big|\mu_0\{x\in\cE_0: \bar S_{0,n}u(x)\leq s\sqrt n\}-\frac1{\sqrt{2\pi\sig^2}}\int_{-\infty}^s e^{-\frac{t^2}{2\sig^2}}dt\big|\leq 
c_\om n^{-\frac12}\ln n
\end{equation}
follows (see Section \ref{ConvR}).

%

Next, as usual, in order to present the local central limit theorem we will distinguish between two 
cases. We will call the case a lattice one if the functions $u_j$ take values at some lattice of the form 
$h\bbZ:=\{hk:\,k\in\bbZ\}$ for some $h>0$. We will call the case a non-lattice if there exist no $h$ which satisfy the latter lattice condition. In the non-lattice case set $h=0$ and $I_h=\bbR\setminus\{0\}$, while 
in the lattice case set $I_h=(-\frac{2\pi}h,\frac{2\pi}h)\setminus\{0\}$.

\begin{theorem}\label{Gen LLT}
Suppose that for any compact interval $J\subset I_h$ we have
\begin{equation}\label{A3}
\lim_{n\to\infty}\sqrt n\sup_{t\in J}\|\cL_{it}^{0,n}\|/\la_{0,n}(0)=0
\end{equation}
and that there exists $c_0>0$ so that $\sig_{0,n}^2=\text{var}_{\mu_0}(S_{0,n}u)\geq c_0n$ for any sufficiently large $n$. Then for any continuous function $g:\bbR\to\bbR$ with compact support we have
\begin{eqnarray*}
\lim_{n\to\infty}\sup_{r\in R_h}\Big|\sqrt{2\pi}\sig_{0,n}\int g(S_{0,n}u(x)-\mu_0(S_{0,n}u)-r)d\mu_0(x)-\\\Big(\int_{-\infty}^{\infty} g(t)dm_h(t)\Big)e^{-\frac{r^2}{2\sig_{0,n}^2}}\Big|=0
\end{eqnarray*}
where in the non-lattice case $R_h=\bbR$ and $m_h$ is the Lebesgue measure, while in the lattice case $R_h=h\bbZ=\{hk:\,k\in\bbZ\}$ and $m_h$ is the measure assigning unit mass to each one of the members of $R_h$.
When the limit $\sig^2=\lim_{n\to\infty}\frac1n\sig_{0,n}^2$ exists then we can replace $\sig_{0,n}$ with $\sig\sqrt n$.
\end{theorem}

Next, let $\cY_1,...,\cY_{m_0}$ be compact metric spaces, set  $\cY_{m_0+1}:=\cY_1$ and let 
$S_j:\cY_j\to\cY_{j+1},\,j=1,2,...,m_0$ be maps so that  Assumption \ref{Ass maps} holds true with $\cE'_j=\cY_{k_j}$ and $T'_j=S_{k_j}$ in place of $\cE_j$ and $T_j$, respectively, where $k=k_j=j\text{ mod }m_0$. Let
$r_i,v_i:\cY_i\to\bbR$ be H\"older continuous functions  whose $\|\cdot\|$ norms are bounded by $B$.
Set $S=S_{m_0}\circ\cdots\circ S_2\circ S_1$ and for each real $t$
let the transfer operator $\textbf{L}_{it}$ be defined by 
\[
\textbf{L}_{it}g(x)=\sum_{y\in S^{-1}\{x\}}e^{\sum_{j=1}^{m_0}r_j(y)+it\sum_{j=1}^{m_0}v_j(y)}g(y).
\]
In Section \ref{LLT sec} we will show that (\ref{A3}) holds true under the following
\begin{assumption}\label{GenPer}
The the spectral radius of $\textbf{L}_{it}$ is strictly less than $1$ for any $t\in I_h$. Moreover,
for any compact interval $J\subset I_h$ there exists $\del_0\in(0,1)$ so that for any constant $B_J>0$ and a sufficiently large $s$ we have 
\begin{equation}\label{GenPerEq}
\lim_{n\to\infty}\frac{|\{0\leq m<n:\,B_J\|\cL_{it}^{m,sm_0}-\textbf{L}_{it}^s\|<1-\del_0\,\,\,\forall t\in J\}|}{\ln n}=\infty
\end{equation}
where $|\Gam|$ is the cardinality of a finite set $\Gam$.
\end{assumption}
Note that in Assumption \ref{GenPer} there is an underlying assumption that $\cE_{m}=\cY_1$ for any $m$ so that $B_J\|\cL_{it}^{m,sm_0}-\textbf{L}_{it}^s\|<1-\del_0$. 
In the non-lattice case, the condition about  the spectral radius of $\textbf{L}_{it}$ means that the function $\sum_{j=1}^{m_0}v_j\circ S_{j-1}\circ\cdots\cdot S_2\circ S_1$ is non-arithmetic (or aperiodic) with respect to the map $S$ defined above, while in the lattice case it means that $h$ satisfies a certain maximality condition with respect to this function (see \cite{GH} and \cite{HH}).

Assumption \ref{GenPer} holds true when $T_k,f_k,u_k$ are chosen at random according to some (not necessarily stationary) sequence of random variables, see Section \ref{NonStat} and a discussion at the end of this Section. 
Non random examples can be constructed as follows: assume that
\[
\|\cL_{it}^{a_k+im_0,m_0}-\textbf{L}_{it}\|\leq\del_k\al(t)
\]
for any sufficiently large $k$ and $0\leq i\leq m_0^{-1}(a_{k+1}-a_k-m_0)$, for some sequence $\del_k$ which converges to $0$ as $k\to\infty$ and a continuous function $\al(t)$, where $(a_k)_{k=1}^\infty$ is a sequence of natural numbers so that $\lim_{k\to\infty}(a_{k+1}-a_{k})=\infty$ and $a_k\leq c_1e^{c_2 k^r}$  for some $r\in(0,1)$ and $c_1,c_2>0$ and all natural $k$'s.


\subsubsection{Exponential concentration inequalities and moderate and large deviations}
Now we will discuss several large deviations type results. We begin with the following 
\begin{theorem}\label{Mart thm}
There exist  constants $C,C_1>0$, which depend only on the initial parameters, so that for any natural $n$ there is a martingale $\{M_j^{(n)}=	W_1^{(n)}+...+W_j^{(n)}:\,j\geq1\}$ whose differences $W_j^{(n)}$ are bounded by $C$, and 
\[
\|S_{0,n}u(\textbf{x}_0)-\bbE[S_{0,n}u(\textbf{x}_0)]-M_n^{(n)}\|_{L^\infty}\leq C_1
\]
where $\textbf{x}_0$ is a random member of $\cE_0$ which is distributed according to $\mu_0$. Moreover, for any $t\geq0$ we have
\[
\mu_0\{x\in\cX:\,|S_{0,n}u(x)-\mu_0(S_{0,n}u)|\geq t+C_1\}\leq 2e^{-\frac{t^2}{4nC}}.
\]
\end{theorem}
By taking $t=\ve n$ we obtain estimates of the form
\[
\mu_0\{x\in\cX:\,|S_{0,n}u(x)-\mu_0(S_{0,n}u)|\geq \ve n\}\leq 2e^{-\frac{\ve^2 n}{16C}}
\]
for any $\ve>0$ and $n\geq1$ so that $\ve n\geq 2C_1$.




The following moderate deviations principle also holds true:
\begin{theorem}\label{Gen MD}
Set $\sig_n^2:=\text{var}_{\mu_0}(S_{0,n}u)$  and suppose that 
\[
\lim_{n\to\infty}\frac{\sig_{n}}{n^{\frac13+\ve_0}}=\infty
\]
for some $0<\ve_0<\frac 16$.
Let $(a_n)_{n=1}^\infty$ be a strictly increasing sequence of real numbers so that $\lim_{n\to\infty}a_n=\infty$
and $\lim_{n\to\infty}n^{-\ve_0}a_n=0$, 
and set $W_n=\frac{S_{0,n}u-\mu_0(S_{0,n}u)}{\sig_n a_n}$. Then for any Borel set $\Gam\subset\bbR$,
\begin{eqnarray}\label{MDP}
-\inf_{x\in\Gamma^0}I(x)\leq \liminf_{n\to\infty}\frac1{a_n^2}\mu_0\{x: W_n(x)\in\Gamma\}\,\,\text{and}\\
 \limsup_{n\to\infty}\frac1{a_n^2}\mu_0\{x: W_n(x)\in\Gamma\}\leq -\inf_{x\in\bar \Gamma}I(x)\nonumber
\end{eqnarray}
where $I(x)=-\frac{x^2}{2}$, $\Gam^o$ is the interior of $\Gamma$ and $\bar\Gamma$ is its closer. 
\end{theorem}

The scaling sequence $(a_n)_{n=1}^\infty$  in theorem \ref{Gen MD} is not optimal even when $\sig_n^2$ grows linearly fast in $n$. In the following circumstances we can also derive more accurate moderate deviations principle together with a local large deviations principle:
\begin{theorem}\label{Gen: Var, LD, MD}
Suppose that for some $\del>0$ for any $z\in B(0,\del)$ the following limit
\[
\Pi(z)=\lim_{n\to\infty}\frac1n\sum_{j=0}^{n-1}\Pi_{j}(z)
\]
exists. Then 
$\Pi(z)$ is analytic on $z\in B(0,2\del)$ and we have
\begin{equation}\label{second exp}
\Pi(z)=\lim_{n\to\infty}\frac1n\ln\mu_0(e^{zS_{0,n}u})=\lim_{n\to\infty}\frac1n\ln l_n(\cL_z^{0,n})
\end{equation}
where under Assumption \ref{Ass maps} $l_n(g)=\sum_{i=1}^{L_n}g(x_{i,n})$ for any $g:\cE_n\to\bbC$, and the $x_{i,n}$'s come from Assumption \ref{Ass maps} (iii), while under either Assumption \ref{Ass maps 2} or Assumption \ref{Ass maps 3} $l_n(g)=g(x_n)$ for an arbitrary point $x_n$.
Moreover:

(i) The limit
\[
p=\lim_{n\to\infty}\frac{1}{n}\int S_{j,n}(x)d\mu_j(x)
\]
exists, and it does not depend on $j$, and when $\mu_j(u_j)=0$ for each $j$ the centralized asymptotic moments (defined in Theorem \ref{MomThm.0}) 
$\gam_k=\lim_{n\to\infty}\gam_{j,k,n}$ exist and they do not depend on $j$. Furthermore, with $\sig^2=\gam_2$ we have
$\gam_k=C_k\sig^k$ for even $k$'s, while for odd $k$'s we have 
$\gam_k=D_k\sig^{k-1}\gam_3$, where the $C_k$'s the the $D_k$'s are defined in Theorem \ref{MomThm.0}.
In addition, 
\[
p=\Pi'(0),\,  \gam_2=\sig^2=\Pi''(0) \,\text{ and }\,\gam_3=\Pi'''(0).
\] 
In particular, in the circumstances of Theorem \ref{Variance SDS} (ii) we have $\sig^2>0$.

(ii) Suppose that $\sig^2>0$. Then for any strictly increasing sequence $(b_n)_{n=1}^{\infty}$ of real numbers 
so that $\lim_{n\to\infty}\frac{b_n}n=0$ and $\lim_{n\to\infty}\frac{b_n}{\sqrt n}=\infty$ and a
Borel set $\Gam\subset\bbR$ (\ref{MDP}) holds true with 
\[
W_n=\frac{S_{0,n}u-\mu_0(S_{0,n}u)}{b_n}
\] 
and $I(x)=\frac{x^2}{2\sig^2}$.

(iii) Let $L(t)$ be the Legendre transform of $\Pi(t)$. Then, (\ref{MDP}) holds true for any Borel set 
 $\Gam\subset[\Pi'(-\del),\Pi'(\del)]$ with $W_n=\frac{S_{0,n}u-\mu_0(S_{0,n}u)}{n}$ and $I(t)=L(t)$.
\end{theorem}
Note that $\Pi'(-\del)<\Pi'(\del)$ when $\sig^2>0$ since then the function $t\to\Pi(t)$ is strictly convex
in some real neighbourhood of the origin.
Theorem \ref{Gen: Var, LD, MD} (ii) is a moderate deviations principle (i.e. with quadratic rate function $I(x)$) which allows scaling sequences $(b_n)_{n=1}^\infty$ of optimal order, as can be viewed from Theorem \ref{Gen: Var, LD, MD} (iii), which is a local large deviations principle. 
 Note also that the equality 
\[
\Pi(z)=\lim_{n\to\infty}\frac1n\ln l_n(\cL_z^{0,n})
\]
can be interpreted for real $z$'s as a sequential analogy of the usual pressure function of the potential $f+zu=f_j+u_j$, in the case of a single map $T$ and functions $f$ and $u$ (see \cite{Bow} in the subshift case). Remark also that, in fact, our proof shows that if one of the limits in (\ref{second exp}) exists, then all of them exists and they are equal.


The limits $\Pi(z)$ does not seem to exist in general, not even for a single $z$. 
When $T_k,f_k,u_k$ are chosen at random according to some (not necessarily stationary) sequence of random variables, we
provide in Section \ref{NonStat} quite general conditions guaranteeing that Assumption \ref{GenPer} holds true and that the limits $\Pi(z)$ exist. In particular, in the circumstances considered in Section \ref{NonStat}, 
the limit $\sig^2=\lim_{n\to\infty}\frac1{n}\text{var}_{\mu_j}(S_{j,n}u)$ exist and it does not depend on $j$. In the circumstances of  Theorem \ref{Variance SDS} (ii), we will derive that $\sig^2>0$. Under certain circumstances, we will also show that (\ref{MomRateTemp}) holds true, and then derive (\ref{BE temp}).





\section{Sequential stability and non-singular maps} 

We will prove  here Theorem \ref{stability}. We will first recall some of the parts of the proof of Theorem \ref{RPF SDS}. The crucial part of the proof is to show that the exists a sequence of complex cones $\cC_j$, containing the constant function $\textbf{1}$, and a constant $n_0$ so that for any $n\geq n_0$ and $j\in\bbZ$ 
\[
\cL_z^{j,n}\cC_j\subset\cC_{j+n}.
\]
Under Assumption \ref{Ass maps}, the linear functional $\te_j(g)=L_j^{-1}\sum_{i=1}^{L_j}g(x_{i,j})$ belongs to the dual of the complex cone $\cC_j$, while under Assumptions \ref{Ass maps 2} and \ref{Ass maps 3} the linear functions $\te_j(g)=g(x_j)$ belongs to this dual, where $x_j$ is an arbitrary point in $\cE_j$.

Next, let $K$, $T_{1,j}$,  $f_{1,j}$ and $u_{1,j}$ be as 
in the statement of Theorem \ref{stability}. Let $\ve>0$, and denote by $\nu_{j,m}$ and $\nu_{1,j,m}$ the $m$-th derivative
at $z=0$ of the maps $z\to\nu_j^{(z)}$ and $z\to\nu_{1,j}^{(z)}$, respectively. Since $\|\nu_j(z)\|$ and 
$\|\nu_{1,j}(z)\|$  are both bounded by some constant $C$, which does not depend on $z$ and $j$, it follows from Lemma 2.8.2 in \cite{book} that for any $z\in\bar B(0,r)$ and $k\geq1$,
\[
\big\|\nu_{j}(z)-\sum_{m=0}^{k}\frac{\nu_{j,m}}{m!}z^m\big\|<(k+2)C2^{-k-1}
\]
and the same inequality holds true with $\nu_{1,j}$ in place of $\nu_{j}$. Let $k=k_\ve$ be the smallest positive integer so that $(k+2)C2^{-k-1}<\frac14\ve$. Then it is sufficient to show that there exists $\del_0>0$ so that 
if $\|\cL_z^{(j)}-\cL_{1,z}^{(j)}\|<\del_0$ for any $j$ and $z\in K$ then for any $m=0,1,...,k_\ve$ we have
\begin{equation}\label{A1}
\|\nu_{j,m}-\nu_{1,j,m}\|<\frac14 r^{-m}k_\ve^{-1}\ve.
\end{equation}
 First, applying (\ref{Exponential convergence}), we get that for any sufficiently large $n$, $z\in U$ and $j\in\bbZ$,
\[
\max\Big(\|\nu_j^{(z)}-F(j,n,z)\|,\,\|\nu_{1,j}^{(z)}-F_1(j,n,z)\|\Big)\leq C_1\del^n
\] 
where the linear functionals $F(j,n,z)$ and $F_1(j,n,z)$ are given by
\[
F(j,n,z)=
\frac{\te_{n+j}(\cL_z^{j,n}(\cdot))}{\te_{n+j}(\cL_z^{j,n}\textbf{1})}\,\,\text{ and }\,\,F_1(j,n,z)=
\frac{\te_{n+j}(\cL_{1,z}^{j,n}(\cdot))}{\te_{n+j}(\cL_{1,z}^{j,n}\textbf{1})}.
\]
Here $C_1$ and $\del\in(0,1)$ depend only on the initial parameters $B,L,\al,\gamma,n_0,\xi$ and $D$ and the linear functional $\te_{n+j}$ was defined at the beginning of this section. Note that the denominators in the definition of $F(j,n,z)$ and $F_1(j,n,z)$ indeed do not vanish since 
$\cL_{z}^{j,n}(\cdot))$ and $\cL_{1,z}^{j,n}(\cdot))$ belong to $\cC_{j+n}$.
Let $n_1=n_1(\ve)$ be the smallest positive integer so that $C_1\del^{n_1}r^{-m}m!<\frac{\ve}8$ for any 
$0\leq m\leq k_\ve$. Then by the Cauchy integral formula, for any $0\leq m\leq k_\ve$ we have
\[
\max\Big(\|\nu_{j,m}-F^{(m)}(j,n_1,0)\|,\,\|\nu_{1,j,m}-F_1^{(m)}(j,n_1,0)\|\Big)\leq r^{-m}m!C_1\del^{n_1}<\frac18\ve
\]
where $F^{(m)}(j,n,0)=\frac{d^mF(j,n,z)}{z^m}\big|_{z=0}$ and $F_1^{(m)}(j,n,z)$ is defined similarly with $F_1(j,n,z)$ in place of $F(j,n,z)$. Observe next that
\[
L_{j+n_1}^{-1}l_{j+n_1}(\cL_0^{j,n_1}\textbf{1})\geq e^{-\|S_{j,n_1}f\|_\infty}
\geq e^{-Bn_1}
\]
and the same inequality holds true with $\cL_{1,0}^{j,n_1}$ in place of $\cL_0^{j,n_1}$.
Therefore, the denominators in $F^{(m)}(j,n_1,0)$ and $F_1^{(m)}(j,n_1,0)$ are bounded from below by 
$e^{-Bn_12^m}$, which depends only on $\ve$ and $m$.
Using that for any families of operators $A_1,...,A_{n}$ and $B_1,...,B_{n}$, we have
\[
A_1\circ A_2\circ\cdots\circ A_{n}-B_1\circ B_2\circ\cdots\circ B_{n}
=\sum_{i=1}^{n-1}A_1\circ\cdots\circ A_{i-1}(A_i-B_i)B_{i+1}\circ\cdots\circ B_{n}
\]
and that $\cL_z^{j,n_1}$ are analytic in $z$ and uniformly bounded in $j$ and $z\in K$ in the operator norm $\|\cdot\|$ (by some constant which depends only on $n_1=n_1(\ve)$ and the initial parameters), we find that 
\[ 
 \sup_{j\in\bbZ}\sup_{z\in K}\|\cL_z^{(j)}-\cL_{1,z}^{(j)}\|<\del_0
\]
for some $\del_0>0$, then for any $j$ and $m\geq0$, 
\[
\|F^{(m)}(j,n_1,0)-F_1^{(m)}(j,n_1,0)\|<C(m,r,\ve)\del_0
\]
where $C(m,r,\ve)$ depends only on $m,\ve, r$ and the initial parameters.
Taking a sufficiently small $\del_0$ completes the proof of the claim about the stability of $\{\nu_j^{(z)}:\,j\in\bbZ\}$ (which was stated in Theorem \ref{stability}).
Since $\la_{j}(z)=\nu_{j+1}^{(z)}(\cL_z^{(j)}\textbf{1})$ and $\la_{1,j}(z)=\nu_{1,j+1}^{(z)}(\cL_{1,j}^{(j)}\textbf{1})$, we drive that for any $\ve>0$ there exists $\del_1>0$ so that 
\[
\sup_{j\in\bbZ}\sup_{z\in K}|\la_j(z)-\la_{1,j}(z)|<\ve\,\, \text{ if }\,\, \sup_{j\in\bbZ}\sup_{z\in K}\|\cL_{z}^{(j)}-\cL_{1,z}^{(j)}\|<\del_1.
\]
Finally, by (\ref{Exponential convergence}) for any $j\in\bbZ$, $n\geq1$ and $z\in U$ we have
\[
\Big\|h_{j}^{(z)}-\frac{\cL_z^{j-n,n}\textbf{1}}{\la_{j-n,n}(z)}\Big\|\leq A\del^n
\]
and similar inequality holds true with $h_{1,j}^{(z)}$, 
\[
\la_{1,j-n,n}(z):=\prod_{i=0}^{n-1}\la_{j-n+i}(z)
\]
and $\cL_{1,z}^{j-n,n}$. Let $n_2$ be so that $ A\del^{n_2}<\frac12\ve$. By taking fixing a sufficiently large $n_2$, we can also assume that $|\la_{j-n_2,n_2}(0)|\geq C_2\|\cL_0^{j-n_2,n_2}\|\geq C_2e^{-Bn_2}$ for some constant $C_2$ which does not depend on $j$.
Using now the stability of  $\la_j(z)$, for any $q>0$ and $m\geq0$ we can find $\del_2=\del_2(q,m,n_2)$ so that 
\[
\Bigg\|\frac{d^mL_j(z)}{dz^m}\Big|_{z=0}-\frac{d^mL_{1,j}(z)}{dz^m}\Big|_{z=0}\Bigg\|<q
\]
if $\|\cL_z^{(j)}-\cL_{1,z}^{(j)}\|<\del_2$ for any integer $j$ and $z\in K$,
where 
\[
L_j(z)=\frac{\cL_z^{j-n_2,n_2}\textbf{1}}{\la_{j-n_2,n_2}(z)}
\]
and $L_{1,j}(z)$ is defined similarly but with $\cL_{1,z}^{j-n_2,n_2}\textbf{1}$ and $\la_{1,j-n_2,n_2}(z)$.
Using this we can approximate each one of the derivatives of $h_{j}^{(z)}$ at $z=0$ by the corresponding derivative of $h_{1,j}^{(z)}$, which, as  in the proof of the stability of $\{\nu_j(z):\,j\in\bbZ\}$, is enough to drive the stability of $\{h_j^{(z)}:\,j\in\bbZ\}$ (for $z\in K$).



\section{Limit theorems: proofs}\label{LimSec}
First, since $\la_j(\cdot)$ are analytic function and $|\la_j(z)|$ is bounded 
uniformly in $j$ and $z$, and $a\leq \la_j(0)\leq b$ for some positive constants $a$ and $b$ which do not depend on $j$,
it is indeed possible to construct  functions $\Pi_j(\cdot)$ which satisfies the
conditions stated in Theorem \ref{MomThm.0} (i).

Next, set 
\[
\tilde\cL_z^{(j)}(g)=\frac{\cL_z(gh_j^{(0)})}{\la_j(0)h_{j+1}^{(0)}}.
\]
Then $\tilde\cL_0^{(j)}\textbf{1}=\textbf{1}$ for any $j$ and  all the results stated in Theorem \ref{RPF SDS} and in the rest of Section \ref{SDSsec} can be applied with the $\tilde\cL_z^{(j)}$. Indeed we can just consider the triplets
\begin{equation}\label{RPF2}
\tilde\la_j(z)=\frac{a_j(z)\la_j(z)}{a_{j+1}(z)\la_j(0)},\,\,
\tilde h_j^{(z)}=\frac{a_j(z)h_j^{(z)}}{h_j^{(0)}}\,\,\, \text{ and }\,\, \,
\tilde\nu_j^{(z)}=(a_j(z))^{-1}h_j^{(0)}\nu_j^{(z)}
\end{equation}
where $a_j(z)=\nu_j^{(z)}(h_j^{(0)})$ (which is nonzero since $h_j^{(0)}\in\cC_{j}$).  
Notice that $\tilde\nu_j^{(0)}=\mu_j$, that $\tilde h_j^{(0)}\equiv 1$ and $\la_j(0)=1$. Moreover, there exists a constant $C$ so that for any $j$ and $z$ we have $|a_j(z)|\leq c$. Since $a_j(0)=0$ it follows that that exist constant $a,b>0$ so that for any $z\in B(0,a)$ we have $b\leq |a_j(z)|\leq c$. Therefore, there exist positive constants $a_1$ $c_1$ so that for any $z\in B(0,a_1)$ and $n\geq1$,
\begin{equation}\label{Press diff}
\big|\sum_{m=0}^{n-1}\Pi_{j+m}(z)-\sum_{m=0}^{n-1}\tilde\Pi_{j+m}(z)\big|=|\ln a_j(z)-\ln a_{j+n}(z)|\leq c_1
\end{equation}
where $\tilde\Pi_j(z),\,j\in\bbZ$ are analytic functions which are defined (simultaneously) in some neighborhood $V$ of the origin, which depends only on the initial parameters, so that $e^{\tilde\Pi_j(z)}=\tilde\la_j(z)$ for each $z$ in $V$.
By the Cauchy integral formula we obtain that for any $k$ there exists a constant $E_k$ so that 
\[
\big|\sum_{m=0}^{n-1}\Pi_{j+m}^{(k)}(0)-\sum_{m=0}^{n-1}\tilde\Pi_{j+m}^{(k)}(0)\big|\leq E_k.
\] 
We conclude that in the proofs of Theorems \ref{Variance SDS}, \ref{Berry-Esseen} and \ref{MomThm.0} we can assume that 
$\cL_z^{(j)}\textbf{1}=\textbf{1}$, $\la_j(0)=1$, $h_j^{(0)}\equiv 1$ and $\nu_j^{(0)}=\mu_j$.

\subsection{The growth of the variance of $S_{0,n}u$} 
We begin with the proof of  Theorem \ref{Variance SDS} (i). It is sufficient to prove this part in the case when $\mu_s(u_s)=0$ for any $s$, for otherwise we will just replace $u_s$ with $u_s-\mu_s(u_s)$. 
First, if there exists a family of functions $\{Y_k:\,k\in\bbZ\}$ as in the statement of Theorem \ref{Variance SDS} (i),
then clearly $\text{var}_{\mu_k}(S_{k,n}u)$ is bounded in $n$, for each integer $k$. On the other hand, 
suppose that 
\[
\liminf_{n\to\infty}\text{var}_{\mu_0}(S_{0,n}u)<\infty.
\]
Since $(T_j)_*\mu_j=\mu_{j+1}$ and $S_{j+1,m}u\circ T_j=S_{j,m+1}u-u_j$ for any $j$ and $m\geq1$, it follows from (\ref{ExpDec}) that 
\[
\liminf_{n\to\infty}\text{var}_{\mu_k}(S_{k,n}u)<\infty 
\]
for any integer $k$. Let $L^{(k)}_\infty$ be the closed linear subspace of $L^2(\cE_k,\mu_k)$ generated
by the functions $u_k,u_{k+1}\circ T_k^1,u_{k+2}\circ T_k^2,u_{k+3}\circ T_k^3,...$.
We conclude that for any $k$ there exists a subsequence of $\text{var}_{\mu_k}(S_{k,n}u)$ 
which converges weakly in $L^2(\cE_k,\mu_k)$ to a member $-Y_k$ of $L^2(\cE_k,\mu_k)$. Using a diagonal argument we can assume that the indexes $(n_m)_{m=1}^\infty$ of the above subsequences do not depend on $k$. 

Next, for any $\zeta\in L^2(\cE_k,\mu_k)$ we have
\[
\lim_{n\to\infty}\mu_k(S_{k,n_m}\cdot\zeta)=-\mu_k(Y_k\zeta).
\]
Therefore, since $\tilde\cL_0^{(k)}$ is the dual of the map $g\to g\circ T_k$ with respect to the spaces $L^2(\cE_k,\mu_k)$
and $L^2(\cE_{k+1},\mu_{k+1})$, we have
\begin{eqnarray*}
\mu_k(Y_{k+1}\circ T_k\cdot\zeta)-\mu_k(Y_k\cdot\zeta)=
\mu_{k+1}(Y_{k+1}\cdot\tilde\cL_0^{(k)}\zeta)-\mu_k(Y_k\cdot\zeta)\\=
-\lim_{m\to\infty}\big(\mu_{k+1}(S_{k+1,n_m}u\cdot\tilde\cL_0^{(k)}\zeta)-\mu_k(S_{k,n_m}u\cdot\zeta)\big)\\=
-\lim_{m\to\infty}\big(\mu_k(S_{k+1,n_m}u\circ T_k\cdot\zeta)-\mu_k(S_{k,n_m}u\cdot\zeta)\big)\\=
\mu_k(u_k\cdot\zeta)-\lim_{m\to\infty}\mu_k(u_{k+n_m}\circ T_k^{n_m}\cdot\zeta)
\end{eqnarray*}
where  we used that $S_{k+1,n_m}u\circ T_k=S_{k,n_{m}+1}u-u_{k}$.
Since $\zeta\in L^{(k)}_\infty$ and $\mu_s(u_s)=0$ for any $s$, it follows from  (\ref{ExpDec}) that 
\[
\lim_{m\to\infty}\mu_k(u_{k+n_m}\circ T_k^{n_m}\cdot\zeta)=0
\]
and hence 
\[
Y_{k+1}\circ T_k-Y_k=u_k
\]
as members of $L^2(\cE_k,\mu_k)$. Now we will show that $Y_k$  can be chosen to be H\"older continuous so that the norms $\|Y_k\|$ are bounded in $k$. 
Let the function $W_k$ be given by 
\[
W_k=\sum_{j=1}^\infty\tilde\cL_0^{k-j,j}u_{k-j}.
\] 
This function is we defined, and is a member of $\cH_k$ because of (\ref{Exponential convergence}), and our
assumption that $\mu_j(u_j)=0$ for any $j$. 
In fact, the exponential convergence  (\ref{Exponential convergence}) implies that $\|W_k\|$ is bounded in $k$. 
Moreover, for each $k$ we have 
\[
W_{k}-\tilde\cL_0^{(k-1)}W_{k-1}=\tilde\cL_0^{(k-1)}u_{k-1}.
\]
Next, substituting both sides of the equality $Y_{k}\circ T_{k-1}-Y_{k-1}=u_{k}$ into $\tilde\cL_0^{(k-1)}$, it follows that
\[
Y_{k}-\tilde\cL_0^{(k-1)}Y_{k-1}=\tilde\cL_0^{(k-1)}u_{k-1}.
\]
Since the functions $\{W_k:\,k\in\bbZ\}$ also satisfy the above
relations, the family of functions $\{d_k:\,k\in\bbZ\}$ given by $d_k=Y_k-W_k$ satisfies that for any $k\in\bbZ$,
\[
d_{k+1}=\tilde\cL_0^{(k)}d_k,\,\,\mu_k-\text{a.s.}.
\]
Since $\tilde\cL_0^{(k)}$ is the dual of $g\to g\circ T_k$
 with respect to the spaces $L^2(\cE_k,\mu_k)$
and $L^2(\cE_{k+1},\mu_{k+1})$, $\tilde\cL_0^{(k)}(v\circ T_k)=v$ for any $v:\cE_{k+1}\to\bbC$ and $(T_k)_*\mu_k=\mu_{k+1}$, it follows that for any integer $k$,
\[
d_{k+1}\circ T_k=d_{k}
\]
or, equivalently, since $Y_{k+1}\circ T_k=Y_k+u_k$,
\[
Y_k-W_k=d_k=Y_{k+1}\circ T_k-W_{k+1}\circ T_k=Y_k+u_k-W_{k+1}\circ T_k.
\]
Therefore, $u_k=W_{k+1}\circ T_k-W_k$. This equality holds true $\mu_k$-almost surely,
but since both sides are continuous and $\mu_k$ assigns positive mass to open sets we derive that it holds true for any point in $\cE_k$. The proof of Theorem \ref{Variance SDS} (i) is complete. 
\qed

Now we will prove Theorem \ref{Variance SDS} (ii). By Theorem \ref{MomThm.0} there exists a constant $R_2$ so that 
\[
\big|\frac{1}{n}\text{var}_{\mu_0}(S_{0,n}u)-\frac1n\sum_{j=0}^{n-1}\la_{j+n}''(0)\big|\leq\frac{R_2}n
\]
where we used that $\la_j''(0)=\Pi_j''(0)$ (recall our assumption that $\mu_j(u_j)=0$). 
Let $T,f,u$ satisfy the conditions stated in Theorem \ref{Variance SDS} (i), and let $(\la(z),h^{(z)},\nu^{(z)})$ be the RPF triplet corresponding to the operators $\cL_z$ given by
\[
\cL_z g(x)=\sum_{y\in T^{-1}\{x\}}e^{f(y)+zu(y)}g(y).
\] 
Then (see \cite{IL}), since $u$ does not admit a co-boundary representation, 
\[
\sig^2=\la''(0)=\lim_{n\to\infty}\frac1n\text{var}_\mu(S_n u)>0.
\]
Note that $\sig^2=\la''(0)\geq c>0$ for some constant $c$ which depends only on the initial parameters
(as Theorem \ref{RPF SDS} holds true also when $T_j,f_j$ and $u_j$ do not depend on $j$).
Let $\ve>0$ and let $\del$ as in Theorem \ref{stability}. 
Then $|\la_j(z)-\la(z)|\leq\ve$ for any $j$, and so by the Cauchy integral formula,
\[
|\la_j''(0)-\la''(0)|=|\la_j''(0)-\sig^2|\leq c\ve 
\]
for some constant $c$ which depends only on the initial parameters.
We conclude that
\[
\Big|\frac1n\sum_{j=0}^{n-1}\la_{j+n}''(0)-\sig^2\Big|\leq
c\ve
\]
and therefore when $c\ve<\sig^2$ we obtain that 
\[
\frac{1}{n}\text{var}_{\mu_0}(S_{0,n}u)\geq \sig^2-c\ve-\frac{R_2}n
\]
which completes the proof of Theorem \ref{Variance SDS} (ii).

\qed

\subsection{Self normalized Berry-Essen theorem} 
Set $\sig_{0,n}=\sqrt{\text{var}_{\mu_0}(S_{0,n}u)}$, and
suppose that 
\[
\lim_{n\to\infty}\sig_{0,n}n^{-\frac13}=\infty.
\]
By replacing $u_j$ with $u_j-\mu_j(u_j)$, we can assume without a loss of generality that 
$\mu_j(u_j)=0$, which implies that
$\mu_j(S_{0,n})=0$ for any $n\geq1$, 
since $(T_j)_*\mu_j=\mu_{j+1}$ for any $j$.
By Theorem \ref{MomThm.0}, 
\begin{equation}\label{Mom11}
\big|\text{var}_{\mu_0}(S_{0,n}u)-\sum_{j=0}^{n-1}\Pi_j''(0)\big|\leq R_2
\end{equation}
for some constant $R_2$, and therefore, 
\begin{equation}\label{Pi growth}
\lim_{n\to\infty}n^{-\frac13}\sum_{j=0}^{n-1}\Pi_j''(0)=\infty.
\end{equation} 
Next, since $\la_j(0)=1$ and $\mu_j=\nu_j^{(0)}$ for any $n\geq1$ and $j$ we have
$(\cL_0^{0,n})^*\mu_{n}=\mu_{0}$
and therefore  for any $z\in\bbC$,
\begin{equation}\label{Fib CLT rel 1}
\mu_{0}(e^{zS_{0,n}u})
=\mu_{n}(\cL_0^{0,n}e^{zS_{0,n}})=
\mu_{n}(\cL_z^{0,n}\textbf{1}).
\end{equation}

Consider now the analytic function $\varphi_{0,n}:U\to\bbC$ given by 
\begin{equation}\label{var phi def 1}
\varphi_{0,n}(z)=\frac{\mu_{n}
(\cL_z^{0,n}\textbf{1})}{\la_{0,n}(z)}=
\int \frac{\cL_z^{0,n}\textbf{1}(x)}{\la_{0,n}(z)}d\mu_{n}(x)
\end{equation}
where $U$ is the neighborhood of $0$ specified in Theorem \ref{RPF SDS}.
Then by (\ref{Fib CLT rel 1}) for any $z\in U$, $j\in\bbZ$ and $n\geq1$,
\begin{equation}\label{Fib CLT rel 2}
\mu_{j}(e^{zS_{0,n}u})=\la_{0,n}(z)
\varphi_{,n}(z)=
e^{\sum_{j=0}^{n-1}\Pi_j(z)}\varphi_{0,n}(z).
\end{equation}
Next, by Theorem \ref{MomThm.0} (ii) we have $\la'_{0,n}(0)=\sum_{j=0}^{n-1}\Pi_j'(0)=\bbE_{\mu_{j}}S_{j,n} u$, and therefore by (\ref{var phi def 1}), 
\begin{equation}\label{var phi ' 0 =0}
\varphi_{0,n}'(0)=0.
\end{equation}

Now, by taking a ball which 
is contained in the neighborhood $U$ specified in Theorem \ref{RPF SDS}, we can always assume that $U=B(0,r_0)$
is a ball around $0$ with radius $r_0>0$.
We claim that there exists a constant $A>0$ such that 
for any $n\in\bbN$ and  $z\in U_0$,
\begin{equation}\label{var phi bound}
|\varphi_{0,n}(z)|\leq A.
\end{equation}
Indeed, by (\ref{Exponential convergence}), there exist 
constants $A_1,k_1>0$ and $c\in(0,1)$ 
such that  for any $z\in U_0$
and $n\geq k_1$,
\begin{equation}\label{L z approx B.E.}
\big\|\frac{\cL_z^{0,n}\textbf{1}}{\la_{0,n}(z)}-h_{n}(z)
\big\|\leq A_1\del^{n}
\end{equation}
 By  Theorem \ref{RPF SDS}, for any $n\geq0$, $j\in\bbZ$ and $z\in U_0$  we have
$\|h_{n}(z)\|\leq C_1$, and  (\ref{var phi bound}) follows.

Next, applying Lemma 2.8.2 in \cite{book} with $k=1$ 
we deduce from  (\ref{var phi ' 0 =0}) and (\ref{var phi bound})
that there exists a constant $B_1>0$ such that 
\begin{equation}\label{The const B}
|\varphi_{0,n}(z)-\varphi_{0,n}(0)|=|\varphi_{0,n}(z)-1|\leq B_1|z|^2
\end{equation}
for any $z\in U_1=B(0,\frac12r_0)=\frac12 U$. 
Moreover, using (\ref{Mom11}) and the above Lemma 2.8.2 there exist constants
$t_0,c_0>0$ such that for any $s\in[-t_0,t_0]$ and a sufficiently large 
$n$,
\begin{equation}\label{Re press0}
\Big|\sum_{j=0}^{n-1}\Pi_j(is)-
\frac{s^2}2\sig_{0,n}^2\Big|\leq c_0|s|^3n+\frac12R_2s^2
\end{equation}
where we also used that 
$
\sum_{j=0}^{n-1}\Pi_j'(0)=\bbE_{\mu_0}S_{0,n}u=0
$
and that $|\Pi_j(z)|\leq C$ for some $C$ which does not depend on $j$ and $z$. 
Set $s_n=\frac{\sig_{0,n}^2}{n2c_0}$ (which is bounded in $n$). Then, by (\ref{Re press0}), there exist constants $q,q_0>0$ so that $q_0s_n\leq \min(t_0,\frac12 r_0)$ and that 
for any sufficiently large $n$ and $s\in[-q_0s_n,q_0s_n]$, 
\begin{equation}\label{Re press1}
\Re\Big(\sum_{j=0}^{n-1}\Pi_j(is)\Big)\leq -qs^2\sig_{0,n}^2.
\end{equation}
Next, by the Esseen-inequality (see \cite{Shir})
for any two distribution functions $F_1:\bbR\to[0,1]$
and $F_2:\bbR\to[0,1]$
with characteristic functions $\psi_1,\psi_2$, respectively, and $T>0$,
\begin{equation}\label{Essen ineq}
\sup_{x\in\bbR}|F_1(x)-F_2(x)|\leq \frac{2}{\pi}\int_{0}^T
\big|\frac{\psi_1(t)-\psi_2(t)}{t}\big|dt+\frac{24}{\pi T}\sup_{x\in\bbR}|F_2'(x)|
\end{equation}
assuming that $F_2$ is a  function with a bounded first
derivative. Set $T_n=\frac{q_0\sig_{0,n}^3}{2c_0 n}$, which converges to $\infty$ as $n\to\infty$. For any real $t$ set 
$t_n=\sig_{0,n}^{-1}t$. Let $t\in[-T_n,T_n]$. Then
by(\ref{Fib CLT rel 2}),
\begin{eqnarray}\label{We have}
|\mu_0(e^{it\sig_{0,n}^{-1}S_{0,n}u})-e^{-\frac12t^2}|\leq e^{\sum_{j=0}^{n-1}\Re(\Pi_j(it_n))}|\varphi_{0,n}(it_n)-1|
\\+|e^{\sum_{j=0}^{n-1}\Re(\Pi_j(it_n))}-e^{-\frac12t^2}|:=I_1(n,t)+I_2(n,t).\nonumber
\end{eqnarray}
By (\ref{Re press1}) we have $e^{\sum_{j=0}^{n-1}\Re(\Pi_j(it_n))}\leq e^{-qt^2}$ and therefore by (\ref{The const B}),
\[
I_1(n,t)\leq B_1e^{-qt^2}t^2\sig_{0,n}^{-2}.
\]
Using the mean value theorem, together with
(\ref{Re press0}) applied with $s=t_n$, taking into account that  $e^{\sum_{j=0}^{n-1}\Re\Pi_j(it_n)}\leq e^{-qt^2}$, we derive that 
\[
I_2(n,t)\leq c_1n\sig_{0,n}^{-3}(|t|^3+t^2)e^{-c_2t^2}
\]
for some constants $c_1,c_2>0$. Let $F_1$ be the distribution function of $S_{0,n}u(\textbf{x}_0)$, where $\textbf{x}_0$ is distributed according to $\mu_0$, and let $F_2$ be the standard normal distribution.
Since the functions $|t|e^{-qt^2}$ and $(|t|^3+t^2)e^{-c_2t^2}$ are integrable, applying (\ref{Essen ineq}) with these functions and the above $T=T_n$ we complete the proof of Theorem \ref{Berry-Esseen}, taking into account that $\sig_{0,n}^{-2}$ converges faster to $0$ than $T_n$.
\qed

\subsection{Local limit theorem and Edgeworth expansion of order three}\label{LLT sec}
Theorem \ref{Gen LLT} follows from the arguments in the proof of Theorem 2.2.3 in \cite{book}, which provides conditions for a  central local limit theorem to hold true in the case when the asymptotic variance $\lim_{n\to\infty}n^{-1}\text{var}(S_{0,n}u)$ exists, and it is positive (using Fourier transforms). Still, 
 when there exist constants $c_1,c_2>0$ so that for any sufficiently large $n$,
\[
\text{var}(S_{0,n}u)\geq c_1n
\]
then, 
by replacing any appearance of $\sig\sqrt n$ with $\text{var}(S_{0,n}u)$, we derive that in order to prove Theorem \ref{Gen LLT} it is sufficient to show that there exists constants $a,b>0$ and $\del\in(0,1)$ so that for any $n\geq1$ and $t\in[-\del,\del]$,
\begin{equation}\label{Goal}
|\mu_0(e^{itS_{0,n}u})|\leq ae^{-bnt^2}.
\end{equation}
When $\sig_{0,n}^2$ grows linearly fast in $n$ then by (\ref{Re press0}), there exists a constant $t_1>0$ so that for any $t\in[-t_1,t_1]$ and a sufficiently large $n$ we have
\begin{equation}\label{Re press}
\sum_{j=0}^{n-1}\Re(\Pi_j(it))\leq -q ns^2
\end{equation}
which together with (\ref{Fib CLT rel 2}) yields (\ref{Goal}).

Now we will show that Assumption \ref{GenPer} implies that (\ref{A3}) holds true. Indeed, using the spectral methods in \cite{HH}, it follows that for any $J\subset I_h$ and $s\geq1$,
\begin{equation}\label{BJ first}
\sup_{t\in J}\|\textbf{L}_{it}^s\|\leq cr^s
\end{equation}
for some $c=c(J)$ and $r=r(J)\in(0,1)$ which may depend on $J$ but not on $s$. Under our assumptions we have $\|\cL_{it}^{j,n}\|\leq 6(1+\frac{2B}{1-\gam^{-1}})\leq B_J$, where $B_J$ is some number which depends only on $J$. Under Assumptions \ref{Ass maps} and \ref{Ass maps 2} such a constant exists due to Lemma 5.6.1 in \cite{book}, while under Assumption \ref{Ass maps 3} it exists in view of (5.3) in \cite{ASIP me}. Write 
\[
\{0\leq m:\,B_J\|\cL_{it}^{m,sm_0}-\textbf{L}_{it}^s\|_\al<1-\del_0\,\,\,\forall t\in J\}=\{m_1<m_2<....\}.
\]
for some strictly increasing  infinite sequence of nonnegative integers $(m_i)_{i=1}^\infty$. 
Fix some $s$ so that $B_Jcr^s<\frac{\del_0}2$, and for each $k$
set $i_k=sm_0k$. Then $m_{i_k}+sm_0\leq m_{i_{k+1}}$ and by (\ref{GenPerEq}) we have
\[
\lim_{n\to\infty}\frac{k_n}{\ln n}=\infty
\]
where $k_n=\max\{k:\,m_{i_k}\leq n-m_0s\}$. Set $l_k=m_{i_k}$ and write: 
\[
\cL_{it}^{0,n}=\cL_{it}^{l_{k_n}+sm_0,n-l_{k_n}-sm_0}\circ\cL_{it}^{l_{k_n},sm_0}\circ\cdots\circ\cL_{it}^{l_2,sm_0}\circ\cL_{it}^{l_1+sm_0,l_2-l_1-sm_0}\circ\cL_{it}^{l_1,sm_0}\circ \cL_{it}^{0,l_1}.
\]
The blocks of the form $\cL_{it}^{l_i,sm_0}$ satisfy 
\[
\|\cL_{it}^{l_i,sm_0}\|_\al<B_J^{-1}(1-\del_0)+cr^s<
B_J^{-1}(1-\frac12\del_0)
\]
and the norm of the other block does not exceed $B_J$. Therefore, for any $n\geq1$ we have
\[
\sup_{t\in J}\|\cL_{it}^{0,n}\|\leq \big(1-\frac12\del_0\big)^{k_n}
\]
and since $\mu_{0}(e^{itS_{0,n}u})=\mu_n(\cL_{it}^{0,n}\textbf{1})$ and $k_n$ grows faster than logarithmically in $n$,
we conclude that (\ref{A3}) holds true.


Next, when (\ref{A3}) holds true, then  the following theorem is proved exactly as in \cite{HafEdge}:
\begin{theorem}
There exists a sequence of polynomials
\[
P_{j,n}(s)=\sum_{k=0}^{m_1}a_{j,n,k}s^j,\,n\geq1
\] 
with random coefficients, whose degree $m_1$ does not depend on $j$ and $n$, so that  for any $n\geq1$ with $\Pi_{j,n,2}=\frac1n\sum_{m=0}^{n-1}\Pi_{j+m}''(0)>0$ we have
\begin{eqnarray*}
\sup_{s\in\bbR}\Big|\sqrt{2\pi}\mu_j\{x\in\cE_j: S_{j,n}(x)\leq\sqrt n s\}-\\\frac{1}{\sqrt{\Pi_{j,n,2}}}\int_{-\infty}^s e^{-\frac{t^2}{2\Pi_{j,n,2}}}dt- n^{-\frac{1}2}P_{j,n}(s)e^{-\frac{s^2}2}\Big|=o(n^{-\frac{1}2}).
\end{eqnarray*}
\end{theorem}
In the case when $T_j,f_j$ and $u_j$ do not depend on $j$, such (and high-order) expansions  were obtained in \cite{Coelho} and \cite{Liverani}.
In \cite{HafEdge} we  obtained high order Edgeworth expansions for random dynamical systems, and, in concrete examples, we have managed to show that the additional required condition  (Assumption 2.4 there) holds true, using some sub-additivity arguments (relying on Kingman's theorem), together with the arguments in \cite{[3]}, which is impossible to generalize for general sequential dynamical systems (as there is no subadditivity of any kind). Therefore we will not formulate here an appropriate theorem, even though it is possible to obtain high-order Edgeworth expansions under appropriate version of the latter Assumption 2.4 (plus additional assumptions in the spirit of Assumption \ref{GenPer}).

\subsection{Martingale approximation, exponential concentration inequalities and moderate deviations}
As we have explained at the beginning of Section \ref{LimSec}, we can assume in course of the proof of Theorem \ref{Mart thm} that $\cL_0^{(j)}\textbf{1}=\textbf{1}$ for each $j$, which means that $\la_j(0)=1$ and $h_j^{(0)}\equiv\textbf{1}$. Moreover, we can clearly assume that $\mu_j(u_j)=0$ for each $j$.
Fix $n\geq1$, and let $\textbf{x}_0$ be a random member of $\cE_0$ which is distributed according to $\mu_0$. Consider the random variables $X_1,X_2,...,X_n$ given by 
\[
X_m=X_m^{(n)}=u_{n-m}(T_0^{n-m}\textbf{x}_0).
\] 
For each $m$, 
consider the $\cE_{m}$-valued random variable $Z_m=T_0^{m}\textbf{x}_0$. Then $Z_{m+1}$ is a function of $Z_{m}$ and so
the family of $\sig$-algebras $\cF_k=\cF_k^{(n)},\,k=1,2,...,n$, 
which is given by 
\[
\cF_{k}=\sig\{Z_n,Z_{n-1},...,Z_{n-k}\}=\sig\{Z_{n-k}\}
\]
is increasing in $k$. For the sake of convenience, set $\cF_m=\cF_n$ when $m>n$.
Next, we claim that. 
\[
\bbE[X_k|\cF_{k-l}]=\cL_0^{n-k,l}(u_{n-k})\circ T_0^{n-k+l}.
\]
Indeed, since for any $j$ we have $(\cL_0^{(j)})^*\mu_{j+1}=\mu_j$
and $(T_j)_*\mu_j=\mu_{j+1}$, 
for any measurable and bounded function $g:\cE_{n-k+l}\to\bbR$ we have
\begin{eqnarray*}
\bbE[g(Z_{n-k+l})\cL_0^{n-k,l}(u_{n-k})\circ T_0^{n-k+l}]=\\\int g(T_0^{n-k+l}x)\big(\cL_{0}^{n-k,l}(u_{n-k})\circ T_0^{n-k+l}\big)(x)d\mu_0(x)\\=
\int g(x)\big(\cL_{0}^{n-k,l}(u_{n-k})\big)(x)d\mu_{n-k+l}(x)=
\mu_{n-k+l}\big(\cL_0^{n-k,l}(g\circ T_{n-k}^l\cdot u_{n-k})\big)\\=
\mu_{n-k}(g\circ T_{n-k}^l\cdot u_{n-k})=\mu_0(g\circ T_0^{n-k+l}\cdot u_{n-k}\circ T_{0}^{n-k})=\bbE[g(Z_{n-k+l})X_k].
\end{eqnarray*}
By (\ref{Exponential convergence}) and since $\mu_j(u_j)=0$ we derive that for each $0\leq l\leq k\leq n$ we have
\[
\|\bbE[X_k|\cF_{k-l}]\|_{L^\infty}\leq C_1\del^l
\]
for some constants $C_1>0$ and $\del\in(0,1)$, which depend only on the initial parameters.
Set $X_m=0$ for any $m>n$ and then for any $j\geq1$ set
\[
W_j=X_j+\sum_{s\geq j+1}\bbE[X_s|\cF_{j+1}]-
\sum_{s\geq j}\bbE[X_s|\cF_{j}].
\]
Then $\{W_j:\,j\geq1\}$ is a martingale difference with respect to the filtration $\{\cF_j:\,j\geq1\}$, whose differences are bounded by some constant $C>0$. Observe that 
\[
\|S_{0,n}u(\textbf{x}_0)-\sum_{j=1}^{n}W_j\|_{L^\infty}
=\|\sum_{j=1}^{n}X_j-\sum_{j=1}^{n}W_j\|_{L^\infty}\leq C_2
\]
for some other constant $C_2$. Set $M_n=\sum_{j=1}^{n}W_j$. Let $(\Om,\cF,P)$ be a probability measure so large so that all 
the random variables defined above are defined on $(\Om,\cF,P)$, and denote by $\bbE_P$ with expectation with respect to $P$. Then  by the Hoeffding-Azuma inequality (see \cite{Mil}), for any $\la>0$ we have
\[
\max\big(\bbE_P[e^{\la M_n}],\bbE_P[e^{-\la M_n}]\big)
\leq e^{\la^2 nC^2}
\]
and so, by the Markov inequality, for any $t\geq 0$,
\begin{eqnarray*}
\mu_0\{x: |S_{0,n}u(x)|\geq C_1+t)=P(|\sum_{j=1}^n X_j|\geq C_1+t)
\leq P(|M_n|\geq t)\\\leq P(M_n\geq t)+P(-M_n\geq t)\leq
P(e^{\la_t M_n}\geq e^{t\la_t})+P(e^{-\la_t M_n}\geq e^{t\la_t})\\\leq 2e^{-t\la_t+\la_t^2nC^2}=2e^{-\frac{t^2}{4Cn}}
\end{eqnarray*}
where $\la_t=\frac{t}{2Cn}$, which together with the previous estimates completes the proof of 
Theorem \ref{Mart thm}.


\subsubsection{\textbf{Moderate deviations theorems via the method of cumulants}}

Relying on (\ref{ExpDec})  and using that $(T_j)_*\mu_j=\mu_{j+1}$, we derive  the following multiple correlation estimate holds: for any $s$ and functions $f_i:\cE_{j+m_i}$, where $i=0,1,...,s$ and $0\leq m_0<m_1<m_2<...<m_s$,
\begin{equation}\label{CumEst}
\Big|\mu_j(\prod_{i=0}^{s}f_i\circ T_j^{m_i})-\prod_{j=0}^{s}\mu_{j+m_i}(f_i)\Big|\leq dM^s\sum_{i=1}^{s}\del^{m_i-m_{i-1}}
\end{equation}
where $d$ is some constant and $M=\max\{\|f_i\|_\al:\,0\leq i\leq s\}$. 

Recall next that the $k$-th cumulant of a (bounded) random variable $W$ is given by 
\[
\Gam_k(W)=\frac{1}{i^k}\frac{d^k}{dt^k}\big(\ln\bbE e^{itW}\big)\big|_{t=0}.
\]
Relying on \ref{CumEst} we can apply Lemma 14 in \cite{Douk}, we derive that
\begin{equation}\label{Comulant est}
|\Gam_k(S_{j,n}u(\textbf{x}_j)-\mu_j(S_{j,n}u))|\leq n(k!)^2c_0^k
\end{equation}
where $\textbf{x}_j$ is distributed according to $\mu_j$ and $c_0$ is some constant. Suppose that 
\[
\lim_{n\to\infty}\frac{\sig_{0,n}}{n^{\frac13+\ve_0}}=\infty
\]
for some $0<\ve_0<\frac16$, where $\sig_{0,n}^2:=\text{var}_{\mu_0}(S_{0,n}u)$. Then, Theorem \ref{Gen MD} follows from  
Theorem 1.1 in \cite{Dor}, applied with $Z_n=\frac{S_{j,n}u\textbf{x}_j-\mu_j(S_{j,n}u)}{\sig_{0,n}}$, $\Del_n=n^{-3\ve_0}$ and $\gam=2$, taking into account that for any $0<\ve_0<\frac16$ and $k\geq3$ we have 
$n^{-k\ve_0}\leq n^{-3\ve_0(k-2)}$.
 Note that several other types of moderate deviations type results follow from the above estimates of the cumulants, see \cite{Dor}.
Remark also that by Corollary 2.1 in \cite{Stat} for any $n\geq1$ we have
\begin{equation}\label{SecBE}
\Big|\mu_0\{x\in\cE_0: S_{0,n}u(x)-\mu_0(S_{0,n}u)\leq r\sig_{0,n}\}-\frac1{\sqrt {2\pi}}\int_{-\infty}^{r}e^{-\frac12 t^2}dt\Big|\leq cn^{-\ve_0}
\end{equation}
for some constant $c$ which depends only on the initial parameters. This provides another proof of the CLT when the variances grow sufficiently fast (note: the rate $n^{-\ve_0}$ is not optimal).

\subsection{Logarithmic moment generating functions}
In this section we will prove Theorem \ref{Gen: Var, LD, MD}. 
Suppose that the limits 
\[
\Pi(z)=\lim_{n\to\infty}\frac1n\sum_{j=0}^{n-1}\Pi_{j}(z)
\]
exist in some open disk $B(0,\del)$ around $0$ in the complex plane. Then $\Pi(z)$ is analytic in $z$ since it is a pointwise limit of a sequence of analytic functions which is uniformly bounded in $n$ (such limits are indeed analytic, as a consequence of the Cauchy integral formula). Next, we claim that we can construct a branch of the logarithm of 
$\mu_0(e^{zS_{0,n}u})$ on $B(0,\del)$ so that
for any $z\in B(0,\del)$,
\begin{equation}\label{Lim Goal}
\Pi(z)=\lim_{n\to\infty}\frac1n\ln\mu_0(e^{zS_{0,n}u}).
\end{equation}
In view of (\ref{Press diff}), we can prove the above in the case when $\cL_z^{(j)}\textbf{1}=\textbf{1}$. In this case, we have 
\[
\mu_{0}(e^{zS_{0,n}u})
=\mu_{n}(\cL_0^{0,n}e^{zS_{0,n}})=
\mu_{n}(\cL_z^{0,n}\textbf{1})
\]
and so by (\ref{Exponential convergence})  we have 
\[
\lim_{n\to\infty}
\Big|\frac{\mu_{0}(e^{zS_{0,n}u})}{\la_{0,n}(z)}-\mu_n(h_n^{(z)})\Big|=0.
\]
Since $h_n^{(0)}=1$ and the norms $\|h_n(z)\|_\al$ are uniformly bounded in $n$ and $z$, there exist positive constants $\del_1, c_1$ and $c_2$ so that for any $z\in B(0,\del_1)$ and  $n\in\bbN$ we have $c_1\leq |\mu_n(h_n^{(z)})|\leq c_2$, which implies that for any sufficiently large $n$,
\[
C_1\leq \Big|\frac{\mu_{0}(e^{zS_{0,n}u})}{\la_{0,n}(z)}\Big|\leq C_2
\]
where $C_1$ and $C_2$ are some positive constants. Therefore, a branch of the logarithm of $\mu_{0}(e^{zS_{0,n}u})$
can be defined so that (\ref{Lim Goal}) holds true.
Note that we also used that $\sum_{j=0}^{n-1}\Pi_j(z)$ is a branch of $\la_{0,j}(z)$. 

In order to prove that 
\[
\Pi(z)=\lim_{n\to\infty}\frac1n\ln l_n(\cL_z^{0,n})
\]
we first observe that  since $\la_j=\nu_{j+1}^{(z)}(\cL_z^{(j)}\textbf{1})$, then by (4.3.26) in \cite{book} we have
\[
|\la_j(z)-g_{j,n}(z)|\leq C\del^n
\]
where
\[
g_{j,n}(z,j)=\frac{l_{n+j+1}(\mathscr{L}_z^{j,n+1}\textbf{1})}{l_{n+j+1}(\mathscr{L}_z^{j+1,n}\textbf{1})}
\]
and $l_{j+n}$ was defined in the statement of Theorem \ref{Gen: Var, LD, MD}.
We also used that the norms $\|\cL_z^{(j)}\|$ are bounded in $j$ and $z\in U$.
Replacing $n$ with $n-j-1$ we obtain that 
\[
\Big|\sum_{j=0}^{n-1}\Pi_j(z)-\sum_{j=0}^{n-1}\ln g_{j,n-j-1}(z,j)\Big|=
\Big|\sum_{j=0}^{n-1}\Pi_j(z)-(\ln l_n(\cL_z^{0,n})-\ln l_0(\cL_z^{(0)})\Big|
\]
is bounded in $n$.

All the statements from Theorem \ref{Gen: Var, LD, MD} (i) follow from Theorem \ref{MomThm.0}. The statements in parts (ii) and (iii) of Theorem \ref{Gen: Var, LD, MD} follow from the Gartner-Ellis theorem (see \cite{Demb}), as noted in \cite{Young3}.

\section{Applicatio to non-stationary random environemnts}\label{NonStat}
The case of independent and identically distributed maps $T_j$ has been widely studied, and extension to stationary maps $T_j$ were then considered. From this point of view it was natural to ask about limit theorems for sequences of map which are not random, but naturally the results that can be obtained are limited, for instance, one could not expect to have a local central limit theorem (without some kind of normalization) in such a generality. This  leads us to consider random maps which are not stationary. In this setup, even the case when the maps are independent but not identically distributed was not studied. Even in this simple case one does not have an underlying skew product map, so the system can not be handled using a single operators as was done in \cite{Aimino} for iid maps.  

\subsection{Random non-stationary mixing environmens: a local CLT}\label{Two sided} 
Let $\{\xi_n:n\in\bbZ\}$ be a family of random variables taking values at some measurable space $\cY$, which are defined on the same probability space $(\Om,\cF,P)$. Let $(\cX,d)$ be a compact metric space and let $\cE\subset\cY\times\cX$ a set measurable with respect to the product $\sig$-algebra, so that the fibers 
 $\cE_y=\{x\in \cX:\,(y,x)\in\cE\},\,\om\in\Om$ are compact. The latter yields (see \cite{CV} Chapter III) that the mapping $y\to\cE_y$ is measurable with respect to the Borel $\sig$-algebra induced by the Hausdorff topology on the space $\cK(\cX)$ of compact subspaces of $\cX$ and the distance function $\text{dist}(x,\cE_y)$ is measurable in $y$ for each $x\in \cX$.  
Furthermore, the projection map $\pi_\cY(y,x)=y$ is
measurable and it maps any $\cF\times\cB$-measurable set to a
$\cF$-measurable set (see ``measurable projection" Theorem III.23 in \cite{CV}).
For each $y\in\cY$, let 
$f_{y},u_{y}:\cE_{y}\to\bbR$ be functions so that the norms $\|f_{y}\|_\al$ and $\|u_{y}\|_\al$ are bounded in $y$. Let $T_{y}$ be  a family of maps from $\cE_{y}$ to $\cX$ so that $P$-a.s. we have 
\begin{equation}\label{Non-Stat ass 1}
T_{\xi_j}:\cE_{\xi_j}\to\cE_{\xi_{j+1}}
\end{equation}
and the family $\{T_j=T_{\xi_j}:\,j\in\bbZ\}$  satisfies all the conditions specified in Section \ref{SDSsec} (with non-random constants) with the spaces $\cE_j=\cE_{\xi_j}$. Note that for (\ref{Non-Stat ass 1}) to hold true, we can just assume that $\cE_{y}=\cX$ for each $y$. 
We also assume here that the maps $(y,x)\to f_y(x), u_y(x)$ and $(y,x)\to T_y(x)$ are measurable with respect to the $\sig$-algebra induced on $\{(y,x):\,y\in\cY,\,x\in\cE_y\}$ from the product $\sig$-algebra on $\cY\times\cX$.
Then by Lemma 5.1.3 in \cite{book}, the norms $\|f_y\|$ and $\|u_y\|$ are measurable functions of $y$.
Let the transfer operators $\cL_z^{(y)}$ which maps functions $g$ on $\cE_y$ to functions on $T_y(\cE_y)$ by the formula
\[
\cL^{(y)}_{z}g(x)=\sum_{a\in T_{y}^{-1}\{x\}}e^{f_{y}(a)+zu_{y}(a)}g(a)
\]
Under Assumption \ref{Ass maps}, 
by Lemma 4.11 in \cite{MSU} there exist $L(y)$ and $x_{i,y}=x_{i,y}$ which are measurable $y$, so that 
\[
\cE_y=\bigcup_{j=1}^{L(y)}B_y(x_{i,y},\xi)
\]
where $B_y(x,\xi)$ is an open open around $x\in\cE_y$ with radius $\xi$. In the above circumstance we also assume that 
$L(y)$ is bounded in $y$ (e.g. when $\cE_{y}=\cX$).
Consider the random operators $\cL^{(j)}_{z}$ given by 
$\cL^{(j)}_{z}=\cL^{(\xi_j)}_{z}$. We remark that the 
the RPF triplets $\la_j(z)$, $h_j^{(z)}$ and $\nu_j^{(z)}$ from Theorem \ref{RPF SDS} are measurable in $\om$ (in view of the limiting expressions of  $\la_j(z)$, $h_j^{(z)}$ and $\nu_j^{(z)}$ from Chapter 4 of \cite{book}).

In the rest of this section, we will impose restrictions on  the process $\{\xi_j\}$, which will guarantee that the 
 local central limit theorem described in \ref{Gen LLT} holds true, where it is sufficient to 
 to derive that Assumption \ref{GenPer} holds true. In order to achieve that,
we will rely on the following
\begin{assumption}\label{Mix Per}
The sequence $\{\xi_j:\,j\in\bbZ\}$ satisfies the following $\phi$-mixing type condition:
there exists a sequence $\phi(n),\,n\geq1$ so that $\sum_{n=1}^\infty\phi(n)<\infty$ and for any $j\in\bbZ$, $n\geq1$, $A\in\sig\{\xi_m:\,m\leq j\}$ and $B\in\sig\{\xi_m:\,m\geq j+n\}$,
\begin{equation}\label{phi}
|P(A\cap B)-P(A)P(B)|\leq P(A)\phi(n)
\end{equation}
where $\sig\{X_i:\,i\in\cI\}$ is the $\sig$-algebra generated by a family of random variables $\{X_i:\,i\in\cI\}$.

(ii) There exist points $y_1,y_2,...,y_{m_0}\in\cY$ so that for any sufficiently large $s\in\bbN$ and all
sufficiently small open neighborhoods $U_i$
of $y_i$,\,$i=1,2,...,m_0$ and $s\geq1$ we have
\begin{equation}\label{PropGrowth}
\lim_{n\to\infty}\frac{\sum_{m=1}^{n}P\big\{\xi_{msm_0+i}\in U_{i};\,\forall\,1\leq i\leq m_0s\}}{\sqrt{n\ln n}}=\infty
\end{equation}
where we set $U_{i}=U_{k_i}$ for any $i>m_0$,  if $i=m_0k_i+r$ some $0\leq r<m_0$.
\end{assumption}
Under Assumption \ref{Mix Per},
set 
\[
\textbf{L}_z=\cL_z^{(y_{m_0})}\circ\cdots\circ\cL_z^{(y_2)}\circ\cL_z^{(y_1)}.
\]

\begin{theorem}\label{LLT RandEnv}
Suppose that Assumption \ref{Mix Per} holds true and that the following two conditions hold true:

(A1)  For any compact $J\subset\bbR$, the family of maps $\{y\to\cL_{it}^{(y)}\}$, where $t$ over $J$, is equicontinuous at the points $y_1,y_2,...,y_{m_0}$, and $\cE_{y}$ does not depend on $y$,  when $y$ lies in some open neighborhood of one of the $y_i$'s.  

(A2) The spectral radius of $\textbf{L}_{it}$ is strictly less than $1$ for any $t\in I_h$, where in the non-lattice case we set $I_h=\bbR\setminus\{0\}$, while in the lattice case we set $I_j=(-\frac{2\pi}{h},\frac{2\pi}{h})$.

Then Assumption \ref{GenPer} holds true. 
\end{theorem}
After showing that  Assumption \ref{GenPer} holds true we can apply Theorem \ref{Gen LLT} with $T_j=T_{\xi_j}, f_j=f_{\xi_j}$ and $u_j=u_{\xi_j}$ when, $P$-a.s. the variance $\text{var}_{\mu_0}(S_{0,n}u)$ grows linearly fast in $n$, which holds true when $\cE_y=\cX$ and $\|\cL_z^{(y)}-\cL_z\|_\al<\ve_0$ for any $y\in\cY$ and complex $z$ in some neighborhood of $0$, where $\cL_z$ and $\ve_0$ are specified in Theorem \ref{Variance SDS} (ii). 

\begin{proof}[Proof of Theorem \ref{LLT RandEnv}]
Let $J\subset$ be a compact set and let $B_J>0$ be as specified after (\ref{BJ first}), so that
that $\|\cL_{it}^{j,n}\|/\la_{j,n}(0)\leq B_J$.
Let $s\in\bbN$ and $\del_0\in(0,1)$. 
 Then, in view of our Assumption (A1), there exist open neighborhoods $U_j$ of $y_j$, $j=1,2,...,m_0$ so that $\|\cL_{it}^{(y)}-\cL_{it}^{(y_j)}\|<\ve_0$ for any $t\in J$, where $\ve_0$ satisfies that $sm_0 B_J^3\ve_0<1-\del_0$. Using that for any families of operators $A_1,...,A_m$ and $B_1,...,B_m$ we have
\begin{eqnarray*}
A_1\circ A_2\circ\cdots\circ A_m-B_1\circ B_2\circ\cdots\circ B_m\\=\sum_{j=1}^m A_1\circ A_2\circ\cdots\circ A_{j-1}(A_j-B_j)B_{j+1}\circ B_{j+2}\circ\cdots\circ B_m
\end{eqnarray*}
we obtain that for any $t\in J$ and $j\in\bbZ$ so that $\xi_{j+k}\in U_{k}$ for any $1\leq k\leq m_0s$,
we have
\begin{equation}\label{So that}
B_J\|\cL_{it}^{j+1,m_0s}-\textbf{L}^{s}_{it}\|_\al\leq sm_0 B_J^3\ve_0<1-\del_0
\end{equation}
where 
\[
\cL_{it}^{j,m_0s}=\cL_{it}^{(\xi_{j+m_0s})}\circ\cdots\circ\cL_{it}^{(\xi_{j+2})}\circ\cL_{it}^{(\xi_{j+1})}.
\]
Next, set 
$
\Gam_{m}=\{\om:\,\xi_{msm_0+i}(\om)\in U_{i};\,\forall\,1\leq i\leq m_0s\}
$
and 
\[
S_n=\sum_{m=1}^{[\frac{n-sm_0}{sm_0}]}\bbI_{\Gam_m}
\]
where $\bbI_\Gam$ is the indicator of a set $\Gam$. Then $S_n$ does not exceed the number of $j$'s between 
$1$ and $n$ so that (\ref{So that}) holds true.
In the above circumstances, we can use Theorem 2.4 in \cite{HafMD} with the random vectors $\{\bbI_{\Gam_{m}}:\,1\leq j\leq sm_0\},\,m=1,2,...$
and derive from (2.11) there that for any $t\geq0$,
\[
P\{|S_n-\bbE S_n|\geq t+c\}\leq 2e^{-c_1\frac{t^2}{n}}
\]
where $c$ and $c_1$ are positive constants which may depend only on $m_0$ and $s$. Taking $t$ of the form $t=t_n=\te\sqrt{n\ln n}$ for an appropriate $\te$, we derive from the Borel-Cantelli lemma that $P$-a.s. for any sufficiently large $n$ we have 
\[
|S_n-\bbE S_n|\leq t_n+c,\,P-\text{a.s.}
\]
which together with (\ref{PropGrowth}) yields that 
\[
\lim_{n\to\infty}\frac{S_n}{\sqrt{n \ln n}}=\infty
\]
and so (\ref{GenPerEq}) holds true $P$-a.s. with the operators $\cL_{it}^{(j)}$, where, in fact, in our circumstances the numerator inside the limit expression in (\ref{GenPerEq}) grows faster than $\sqrt{n \ln n}$.
\end{proof}


\subsection{Examples}



We will provide here several examples in which Assumption  \ref{Mix Per} is satisfied.
First, when considering the simple case when $\cY=\bbZ$, then for any $U_1\subset(-1,1)$ which contains $0$ we have
\[
\sum_{m=1}^{n}P\big\{\xi_{mm_0s+i}\in U_{i};\,\forall\,1\leq i\leq m_0s\}
=\sum_{m=1}^{n}P\{\xi_{mm_0s+i}=y_{i};\,\forall\,1\leq i\leq m_0s\}
\]
where we set $y_i=y_r$ if $i=m_0k+r$ for some integers $k\geq 0$ and $0\leq r<m_0$.
When $\xi_i$'s form an inhomogenious Markov chain (e.g. when they are independent) with $n$-th step transition probabilities $p^{(i,n)}_{j,k}=P(\xi_{i+n}=k|\xi_i=j)$ (so that (\ref{phi}) holds true), then we only require that 
\[
\sum_{m=1}^{n}\prod_{i=1}^{m_0s}P(\xi_{mm_0s}=y_1)p^{(mm_0s+i,1)}_{y_i,y_{i+1}}
\] 
grows faster than $\sqrt{n \ln n}$ in $n$, for any $s$. 
 For instance, we could require that 
 $p^{(j,1)}_{y_i,y_{i+1}}\geq\del_j$ and $P(\xi_j=y_1)\geq r_j$ for some $\del_j,r_j>0$ and all $i$'s, which will give us linear growth if $\del_i$'s and $r_j$'s are bounded from below by some positive constant, while in general we can impose certain restrictions on the $\del_i$ and $r_j$'s to obtain the desired growth rate (e.g when $\frac1{\del_j}$ is at least of logarithmic order in $j$ and $r_j$ decays sufficiently slow to $0$ as $j\to\infty$).
Note that in these circumstances condition (i) from Theorem \ref{LLT RandEnv} trivially holds true since $\cL_z^{(y)}$ are locally constant in $y$ around $y_1,...,y_{m_0}$.

A close but more general situation is the case when the maps $y\to T_y,f_y,u_y$ are locally constant around 
$y_1,y_2,...,y_{m_0}$ and $\{\xi_j\}$ is an inhomogeneous Markov chain so that, 
\[
P(\xi_{i+1}\in U|\xi_i=x)=\int_{U} p_i(x,y)d\eta(y),\,\,\forall x\in V
\]
for any sufficiently small neighborhoods $U$ and $V$ some $y_i$ and $y_j$, respectively.
Here $\eta$ is some probability measure on $\cY$ which assigns positive mass to open sets, and $p_i(x,y)$ are functions which are bounded from below by some positive constants $\del_i$. In this case we have
\[
\sum_{m=1}^n P\big\{\xi_{mm_0s+i}\in U_{i};\,\forall\,1\leq i\leq m_0s\}\geq \big(\prod_{j=1}^{m_0}\eta(U_j)\big)^{s}
\sum_{m=1}^{n}\prod_{i=1}^{m_0s}\del_{mm_0s+i}.
\]
Imposing some restrictions on the $\del_i$'s we will get that the above right hand side grows faster than $\sqrt{n\ln n}$. For instance, when $\frac1{\del_j}\leq c\ln j$ then, for any $s$,  
$
\prod_{i=1}^{m_0s}\del_{mm_0s+i}
$
is of order $\frac1{\ln^\te m}$ in $m$,  where $\te=m_0s$, and so 
\[
\sum_{m=1}^{n}\prod_{i=1}^{m_0s}\del_{mm_0s+i}
\]
is at least of order $\frac{n}{\ln^\te n}$ in $n$. When $\cY$ is compact and the densities $p_i(x,y)$ are bounded from below and above by some positive constant then  condition (\ref{phi}) holds true with $\phi(n)$ of the form $\phi(n)=ae^{-nb}$ for some $a,b>0$ (this follows from Section 6 in \cite{book} and the arguments in the proof of Lemma 7.1.1. in \cite{FunCLT}), and in this case the $\del_j$'s are bounded from below.
Still, relying on (\ref{phi}) our arguments allow that $\liminf_{i\to\infty}\del_i=0$'s,  as $i\to\infty$, as described above.

Finally, let $(\cQ,T,\textbf{m})$ be a mixing Young tower (see \cite{Young1} and \cite{Young2}) whose tails $\nu\{R\geq n\}$ are of order $n^{-a}$ for some $a>0$ (here $\nu$ is the original measure on the tower, $R$ is the return time function 
and $\textbf{m}$ is the invariant mixing probability measure).
Consider the case when $\xi_j(q)=H_j(T^jq)$, where $q$ is distributed according to $\textbf{m}$ and $H_j$ is a function that is constant on elements of the partition defining the tower. Then by
(7.6) in \cite{FunCLT}  the inequality (\ref{phi}) holds true with a summuable sequence $\phi(n)$ 
if the tails of the tower decay polynomially and sufficiently fast, and the function $H_j$ is measurable with respect to the partition which defines the tower.
 Next, suppose that $T$ has a periodic point $q_0$ with period $m_0$ and that $H_j(T^mq_0):=y_m$ does not depend on $j$ for each $m=0,1,...,m_0-1$. We will show now that (\ref{PropGrowth}) holds true with the above $y_i$'s. Let $U_i$ be a neighborhood of $y_i$, where $i=0,1,...m_0-1$, and let $s\geq 1$. Observe that $(\xi_{msm_0+i}(q))_{i=1}^{sm_0}$ takes the value $(y_1,y_2,...,y_{0})^{\otimes s}$ when $q$ lies in a set of the form $T^{-msm_0}A_s$, for some open  neighborhood $A_s$ of the periodic point $q_0$. Here the power $s$ stands for concatenation: $a^{\otimes s}=aaa...a$.
Since $\textbf{m}$ is $T$-invariant we derive that
\[
\sum_{m=1}^{n}P\big\{\xi_{msm_0+i}\in U_{i};\,\forall\,1\leq i\leq m_0s\}\geq 
n\textbf{m}(A_s)
\]
and so (\ref{PropGrowth}) holds true (as $\textbf{m}(A_s)>0$). We note that the case when all of the $H_j$'s are H\"older continuous uniformly in $j$ can be considered, since then we can approximate (in the proof of Theorem \ref{LLT RandEnv}) the $H_j$'s by functions which are constant on the above partitions, and use again (7.6) in \cite{FunCLT} .

We can also consider the case when each $\xi_j$ depends only the first coordinate in the following model:

\paragraph{\textbf{Non-stationary subshifts of finite type}.}
Let $d_j,\,j\in\bbZ$ be a family of positive integers so that $d_j\leq d$ for some $d\in\bbN$ and all $j$'s. Let $A_j=A_j(a,b)$ be a family of matrices of sizes $d_j\times d_{j+1}$ whose entries are either $0$ or $1$, so that all the entries of $A_{j+n_0}\cdots A_{j+n_0-1}\cdot A_{j+1}$ are positive, for some $n_0\geq1$ and all $j$'s.
 Let the compact space $\cX$ be given by 
\[
\cX=\{1,2,...,d\}^{\bbN\cup\{0\}}
\] 
and for each $j$, and let $d(x,y)=2^{-\min\{n\geq0: x_n\not=y_n\}}$.
For each integer $j$ set
\[
\cE_j=\{(x_{j+m})_{m=0}^\infty\in\cX:\,x_{j+m}\leq d_{j+m}\,\text{ and }\, A_{j+m}(x_{j+m},x_{j+m+1})=1,\,\,\forall m\geq 0\}
\]
and define $T_j:\cE_j\to\cE_{j+1}$ by 
\[
T_j(x_{j},x_{j+1},x_{j+2},..)=(x_{j+1},x_{j+2},...).
\]
We also set $\gam=2$ and $\xi=1$, so the inequality $d(x,y)<\xi$ means that $x_0=y_0$. In this nonstationary subshift case we have the following result: 

\begin{theorem}\label{SFT pros}
(i) There exist constants $C_1,C_2>0$ so that fir any symbols $a_j,...,a_{j+r}$, where $j\in\bbZ$ and $r\geq0$ we have
\[
C_1\leq \mu_j\{(x_m)_{m=j}^\infty: x_i=a_i\,\,\,\forall \in[j,j+r]\}/e^{S_j^n\phi(a^{(j)})-\ln\la_{j,r}(0)}\leq C_2
\]
where $a^{(j)}\in\cX_j$ is any point so that $a^{(j)}_i=a_i$ for any $j\leq i\leq j+r$.

(ii) There exist constants $C>0$ and $\del\in(0,1)$ so that for
any integer $j$, $r,s,n\geq1$, symbols $a_j,...,a_{j+r}$ and $b_{j+r+n},...,b_{j+r+n+s}$ and 
cylinder sets
\[
A=\{(x_m)_{m=j}^\infty: x_i=a_i\,\,\forall \in[j,j+r]\}\subset\cE_j
\] 
and 
\[
B=\{(x_m)_{m=j}^\infty: x_i=b_i\,\,\forall \leq i\in[j+r+n, j+r+n+s]\}\subset\cE_{j}
\]
we have
\begin{equation}\label{SFT mix}
|\mu_j(A\cap B)-\mu_j(A)\mu_{j}(B)|\leq C\mu_j(A)\mu_{j}(B)\del^{n}.
\end{equation}
Namely, the $\sig$-algebras generated by the cylinder sets are exponentially fast $\psi$-mixing (uniformly in $j$).
\end{theorem}
This result is proved similarly to \cite{Bow}.



\subsection{Random sequential dynamical environments}\label{RandSeqDyEnv}

\subsubsection{\textbf{Random sequential distance expanding environments}} 
Let $\cY$, $\cE_y,T_y,f_y$ and $u_y$ be as in the beginning of Section \ref{Two sided}. 
We assume here that $\cY$ is a metric space and let $\cB$ be its Borel $\sig$-algebra. Let $\cY_j\subset\cY,\,j\geq0$ be a family of closed sets, $\te_j:\cY_{j}\to\cY_{j+1},\,j\in\bbZ$ be family of measurable maps and $P_j,\,j\in\bbZ$ be family of probability measures on $\cY$ which are supported on $\cY_j$, respectively, so that $(\te_j)_*P_j=P_{j+1}$ for each $j$.
For each $j$ and $m\geq1$ set $\te_j^m=\te_{j+m-1}\circ\cdots\circ\te_{j+1}\circ\te_j$ and 
consider the case when $\zeta_j=\te_0^j\textbf{y}_0$, where $\textbf{y}_0$ is distributed according to $P_0$.
In this section we will start with a certain type of one sided sequences  and consider iterates of the form 
\[
\cL_0^{(\zeta,n)}:=\cL_0^{(\zeta_{j-1})}\circ\cdots\circ\cL_0^{(\zeta_1)}\circ\cL_0^{(\zeta_0)}.
\] 
Namely, we view here the $\zeta_j$'s as a sequential dynamical random environment and consider (random) one sided sequences of maps $T_j$ and functions $f_j$ and $u_j$ of the form 
\[
T_j=T_{\zeta_j},\,f_j=f_{\zeta_j}\,\,\text{ and }\,\,u_j=u_{\zeta_j}.
\]
Note that we can not apply directly Theorem \ref{RPF SDS} and all the other results from Section \ref{SDSsec}  since we only have one sided sequences. In order to overcome this difficulty we will need the following.
Let $\hat\cY$ be the space of all sequences $\hat y=(y_k)_{k=-\infty}^\infty\in\cY^\bbZ$ so that $y_{k+1}=S_ky_k$ for each $k$, let $\sig:\hat\cY\to\cY$ be the shift map given by $\sig x=(x_{k+1})_{k=-\infty}^\infty$ and set $\xi_k=\sig^k\xi_0$ where $\xi_0$ is distributed according to the measure $\hat P$ induced on $\hat\cY$ by 
the sequence of finite dimensional distributions given by 
\[
\hat P_k\{y: y_i\in A_i;\,\,\forall -k\leq i\leq k\}=P_{-k}\big\{\bigcap_{i=-k}^{k}(\te_{-k}^{i})^{-1}A_i\big\}.
\]
Note that the Kolmogorov extension theorem indeed can be applied (i.e. the family $\{\hat P_k\}$ is consistent) since $(\te_j)_*P_j=P_{j+1}$ for each $j$. Henceforth, we will refer to the process $\{\xi_j:\,j\in\bbZ\}$ as the ``invertible extension" of the process $\{\zeta_j:\,j\geq0\}$. Now, we can view $T_j,f_j$ and $u_j$ as functions of $\xi_j$: they depend only on the $0$-th coordinate of $\xi_j$ (so now $T_j$, $f_j$ and $u_j$ are defined also for negative $j$'s). Henceforth, $\{\la_j(z):\,j\in\bbZ\}$, $\{h_j^{(z)}:\,j\in\bbZ\}$ and $\{\nu_j^{(z)}:\,j\in\bbZ\}$ will denote the
RPF triplets corresponding to the (random) family of operators $\cL_z^{(j)}=\cL_z^{(\xi_j)}:=\cL_z^{(\pi_0\xi_j)}$, where $\pi_0 y=y_0$.

In the following section we will provide general conditions under which the results stated in Section \ref{SDSsec} hold true. Note that formally, we will show that the limit theorems stated in Section \ref{SDSsec} hold true with the extension 
$\{\xi_j:\,j\in\bbZ\}$ as the random environment, but when the random Gibbs measure $\mu_j$ given by $d\mu_j=h_j^{(0)}d\nu_j^{(0)}$ depend only on $\xi=\{\zeta_j:\,j\geq0\}$ then we can formulate all the limit theorems without passing to the invertible extension. First, taking a careful look at the arguments in Chapter 4 of \cite{book}, we see that the functional $\nu_j^{(z)}$ depend only on $\zeta_j$. Therefore, $\mu_j$ depends only on $\xi_j$ if the function $h_j^{(0)}$ is deterministic. We refer the readers to Theorem \ref{II'} (and its proof) for conditions which guarantee that $h_j^{(0)}=h$ for any $j$, for some deterministic function $h$ (take there $T_j$ in place of $S_j$).

\subsubsection{\textbf{The LCLT}}
We assume here that for any Lipschitz function $g$ on $\cY$, an integer $s\geq1$ and $t\geq0$ we have
\begin{equation}\label{Lip Corr Ass}
P_0\big\{\big|\sum_{j=0}^{n-1}g\circ\te_0^{js}-\sum_{j=0}^{n-1}\bbE_{P_0}g\circ\te_0^{js}\big|\geq t+c_1\big\}\leq c_2e^{-c_3\frac{t^2}{n}}
\end{equation}
where $c_1, c_2$ and $c_3$ are some positive constants which may depend on $g$. The inequality (\ref{Lip Corr Ass}) holds true when $\te_j$'s are maps satisfying  the assumptions from Section \ref{SDSsec}  and $P_j$'s are the appropriate (sequential) Gibbs measures corresponding to these maps.

In this section we will prove the following:
\begin{theorem}\label{III}
Suppose that (\ref{Lip Corr Ass}) holds true, that there exists a point $y_0\in\bigcap_{j=0}^\infty\cY_j$ so that $\te_jy_0=y_0$ for each $j$ and that 
the  maps $\te_j$ are H\"older continuous with exponent $\be\in(0,1]$ and H\"older constant less or equal to $K$, for some constants $\be$ and $K$ which do not depend on $j$. Assume, in addition, that
for any open neighbourhood $V$ of $y_0$ we have 
\begin{equation}\label{PerII}
\lim_{n\to\infty}\frac{\sum_{j=1}^{n}P_{j}(V)}{\sqrt{ n\ln n}}=\infty
\end{equation}
and that conditions (A1) and (A2) from Theorem \ref{LLT RandEnv} hold true with $\textbf{L}_{it}=\cL_{it}^{(y_0)}$.
Then Assumption \ref{GenPerEq} holds true $P_0$-a.s. with $m_0=1$ and the above $\textbf{L}_{it}$.
\end{theorem}

Note that when verifying Assumption \ref{GenPerEq}, we can use the one sided sequence $\{\zeta_j\}$ without passing to its invertible extension.
When $\te_j,\,j\geq0$ are distance expanding maps which satisfy one of  Assumptions \ref{Ass maps}-\ref{Ass maps 3}, assuming that all of them have the same fixed point, then we can take $P_j$ to be the appropriate $j$-th Gibbs measure. In this case, by Lemma 5.10.3 in \cite{book} we have $P_j\big(B_j(y_0,r)\big)\geq Cr^q$ for some $q>0$, which implies that  (\ref{PerII}) holds true, since in this case the numerator grows linearly fast in $n$. Note that a common fixed point exists, for instance when all $\cY_j$ are, the same torus and all $S_j$'s vanish at the origin, and when $S_j$'s form a non-stationary subshift of finite type  so  that the matrices $A_j,\,j\geq0$ defining the shift satisfy that $A_j(a,a)=1$ for some $a\in\bbN$ (and then we can take $y_0=(a,a,a,...)$).

\begin{proof}[Proof of Theorem \ref{III}]
Let $s\geq1$.
As in the proof of Theorem \ref{LLT RandEnv}, it is sufficient to show that for any neighborhood $U$ of $y_0$ in 
$\cY$ we have
\[
\lim_{n\to\infty}\frac{\sum_{m=0}^{n-s}\bbI_{\Gam_m}(\om)}{\sqrt{n\ln n}}=\infty,\,P-\text{a.s.}
\]
where $\Gam_m=\{\om:\,(\zeta_{ms}(\om),\zeta_{ms+1}(\om),...,\zeta_{ms+s-1}(\om))\in U\times U\times\cdots\times U\}$, and $\bbI_A$ is the indicator function of a set $A$.
Equivalently, we need to show that for $P_0$-almost any $y$ we have  
\[
\lim_{n\to\infty}\frac{\sum_{m=0}^{n-s}\bbI_{\Del_m}(y)}{\sqrt{n\ln n}}=\infty
\]
where 
\[
\Del_m=\bigcap_{j=0}^{s-1}(\te_0^{ms+j})^{-1}U=(\te_{0}^{ms})^{-1}U_{m,s}
\]
and $U_{m,s}=U\cap\te_{ms}^{-1}U\cap(\te_{ms}^2)^{-1}U\cap...\cap(\te_{ms}^{s-1})^{-1}U$. 
Since $y_0$ is a common fixed point, for any $r>0$ we have
\[
B_m(y_0,r^{\frac 1\beta}K^{-\frac 1\beta})\subset\te_m^{-1}B_{m+1}(y_0,r)
\]
where for each $m$, $x\in\cY_m$ and $\del>0$ the set $B_m(x,\del)$ denotes an open ball in $\cY_m$ around $x$ with radius $\del$.
Therefore, $U_{m,s}$ contains an open ball $V_{m,s}=B_m(y_0,2r_s)=B(y_0,2r_s)\cap\cY:=V_s$ around $y_0$ in $\cY_m$, whose radius does not depend on $m$ (here $B(y_0,2r_s)$ is the corresponding ball in $\cY$). Hence, it is sufficient to show that 
\begin{equation}\label{Need}
\lim_{n\to\infty}\frac{\sum_{m=0}^{n-s}\bbI_{V_s}\circ \te_0^{ms}}{\sqrt{n\ln n}}=\infty,\,P_0-\text{a.s.}.
\end{equation}
Let $f$ be Lipschitz function so that 
\[
\bbI_{B(0,r_s)}\leq f\leq\bbI_{V_s}=\bbI_{B(0,2r_s)}.
\] 
Then 
\[
\sum_{m=0}^{n-s}\bbI_{V_s}\circ \te_0^{ms}\geq\sum_{m=0}^{n-s}f\circ\te_0^{ms}.
\]
Note that $\bbE_{P_0}f\circ\te_0^{ms}=\bbE_{P_{ms}}f=\int fdP_{ms}\geq P_{ms}(B(0,r_s))$.
Taking $t$ of the form $t=t_n=c\sqrt{n\ln n}$ in  (\ref{Lip Corr Ass}), for an appropriate $c$, and using the Borel-Cantelli lemma we derive that $P_0$-a.s. for any sufficiently large $n$,
\[
\big|\sum_{m=0}^{n-s}f\circ\te_0^{ms}-\sum_{m=0}^{n-s}\bbE_{P_0}f\circ\te_0^{ms}\big|\leq C \sqrt{n\ln n}
\]
where $C$ is some constant. The letter inequality together with the former inequalities and (\ref{PerII}) imply (\ref{Need}).
\end{proof}

\begin{remark}
The proof of Theorem \ref{III} proceeds similarly if we assume that there exists $y_0,y_1,...,y_{m_0-1}\in\cY$ so that for any $j$ and $i=0,1,...,m_0-1$ we have $S_{j+i}y_i=y_{i+1}$, where $y_{m_0}:=y_0$, that (\ref{PropGrowth}) holds true, with $y_{i-1}$ in place of $y_i$ appearing there and with $\zeta_{mm_0s+i}$ in place of $\xi_{mm_0s+i}$, 
and that for any family of Lipschitz functions $g_j,\,j\geq0$ with Lipschitz constant less or equal to $1$ so that 
$\sup|g_j|\leq1$ we have
\begin{equation}\label{Lip Corr Ass.1}
P_0\big\{\big|\sum_{j=0}^{n-1}g_j\circ\te_0^{j}-\sum_{j=0}^{n-1}\bbE_{P_0}g_j\circ\te_0^{js}\big|\geq t+c_1\big\}\leq c_2e^{-c_3\frac{t^2}{n}}
\end{equation}
for some positive constants $c_1$ and $c_2$ (which do not depend on the $g_j$'s). In this case we require that the spectral radius of 
\[
\textbf{L}_{it}=\cL_{it}^{(y_{m_0-1})}\circ\cdots\circ\cL_{it}^{(y_1)}\circ\cL_{it}^{(y_0)}
\]
is less than $1$ for any $t\in I_h$. 
In particular we can consider non-stationary subshifts of finite type, with the property that for some $a_1,a_2,...a_{m_0}\in\bbN$ and all $j$'s we have $A_j(a_i,a_{i+1})=1$, where $a_{m_0+1}:=a_1$, which means that the periodic word $(a_1,a_2,a_3,...a_{m_0},a_1,a_2,...,a_{m_0},...)=(a_1,a_2,...,a_{m_0})^{\otimes\bbN}$ belongs to all of the $\cE_j$'s.
\end{remark}

\subsubsection{\textbf{Existence of limiting logarithmic moment generating functions}}
We assume here that $\cY$ is a compact metric space.
Let $S_j:\cY\to\cY$ be a family of maps satisfying all the conditions specified in  Section \ref{SDSsec} with $\cE_j=\cY$, and consider the case when $\te_j=S_j$ for each $j$.
Let $r_j:\cY\to\bbR$ be a family of maps so that the H\"older norms $\|r_j\|_\al$  are bounded in $j$, and let 
$(\boldsymbol{\la}_j(0),\textbf{h}_j^{(0)},\boldsymbol{\nu}_j^{(0)})$ be the RPF triplet corresponding to the operators $\mathscr{L}_0^{(j)}$ given by 
\[
\mathscr{L}_0^{(j)}g(x)=\sum_{y\in S_j^{-1}\{x\}}e^{r_j(y)}g(y).
\]  
In these circumstances, we take $P_j=\boldsymbol{\mu}_j$, where $\boldsymbol{\mu}_j$ is given by 
$d\boldsymbol{\mu}_j=\textbf{h}_j^{(0)}d\boldsymbol{\nu}_j^{(0)}$. Namely, we consider here a random sequential 
environment generated by a two sequence of distance expanding maps $S_j$.

We will first need the following
\begin{theorem}\label{I}
(i) The random pressure function
$\Pi_j(z)$ depends only on $\zeta_j$.

(ii) When the maps $S_j$ are H\"older continuous with the same exponent and H\"older constants which are bounded in $j$
then $P$-a.s. we have 
\[
\lim_{n\to\infty}\Big|\frac1n\sum_{j=0}^{n-1}\Pi_j(z)-\frac1n\sum_{j=0}^{n-1}\bbE_P\Pi_j(z)\Big|=0.
\] 
In particular,
when the distribution of $\zeta_j$ does not depend on $j$ (i.e, when $P_j=\boldsymbol{\mu}_j$ does not depend on $j$) then the limit $\Pi(z)$ from Theorem  \ref{Gen: Var, LD, MD}
exists and $\Pi(z)=\bbE_P\Pi_0(z)$.
\end{theorem}
The distribution of $\zeta_j$ does not depend on $j$ when, for instance, each one of the $S_j$'s has the form $S_j(x)=(m_1^{(j)}x_1,...,m_d^{(j)}x_d)\text{mod }1$ for some positive integers $m_i^{(j)}$, where $x=(x_1,...,x_d)\in\cE_j=\bbT^d$ belongs to some $d$-dimensional Torus (here $\nu_j^{(0)}=\text{Lebesgue}$ and $\textbf{h}_j^{(0)}\equiv1$).
More generally, 
when there exists  a probably measure $\textbf{n}$  on $\cY$ which assigns positive mass to open sets, so that $(S_j)_*\textbf{n}\ll\textbf{n}$ for any integer $j$ and  $r_j$ be defined by $e^{-r_j}=\frac{d(S_j)\textbf{n}}{d\textbf{n}}$, then by Theorem \ref{NonSinThm} we have $\boldsymbol{\la}_j(0)=1$ and $\boldsymbol{\nu}_j^{(0)}=\textbf{n}$. In this case, the claim that the distribution of $\zeta_j$ does not depend on $j$ means that $\textbf{h}_j^{(0)},\,j\geq0$ does not depend on $j$, and in Theorem \ref{II'} we will show that $\textbf{h}_j^{(0)}$ does not depend on $j$ and $\om$ when the $S_j$'s are drawn at random according to some, not necessarily stationary, classes of processes (so we will have "random random" non-stationary neighborhoods).

\begin{proof}[Proof of Theorem \ref{I}]
Set 
\[
g_{j,n}(z,\zeta_j)=\frac{l_{n+j+1}(\mathscr{L}_z^{j,n+1}\textbf{1})}{l_{n+j+1}(\mathscr{L}_z^{j+1,n}\textbf{1})}.
\]
Since $\la_j=\nu_{j+1}^{(z)}\cL_z^{(j)}\textbf{1}$ and the norms $\|\cL_z^{(j)}\|_\al$ are bounded in $j$ and $z\in U$,
applying (\ref{Exponential convergence}) we obtain that for any $n\geq1$ and $j\in\bbZ$,
\begin{equation}
|\la_j(z)-g_{j,n}(z,\zeta_j)|\leq C\del^n
\end{equation}
where $C>0$ and $\del\in(0,1)$ are constants,
and therefore $\la_j(z)$ and $\Pi_j(z)$ depend only on $\zeta_j$.
We also derive from the above that there exists a constant $r_2>0$ so that for any $j$ and complex $z$ so that $|z|<r_2$ we have 
\begin{equation}\label{Approx}
|\Pi_j(z)-\ln g_{j,n}(z,\zeta_j)|\leq C_1\del^n
\end{equation}
where $C_1>0$ is some constant we used that $\la_j(0)\geq a$ for some constant $a>0$. Set 
\[
\hat g_{j,n}(z,\zeta_j)=\ln g_{j,n}(z,\zeta_j).
\]
Let $\ve>0$ and let $k$ be so that $C\del^k<\ve$. In the circumstances of Theorem \ref{I} (i), the maps $g_{j,k}(z,\zeta_j)$ are H\"older continuous functions of $\zeta_j$ with the same exponent and with H\"older constants which are bounded in $j$.
Therefore, by Theorem \ref{Mart thm}, there exist constants $c_1,c_2$ and $c_3$, which may depend on $k$, so that for any $\ve>0$ we have 
\begin{equation}\label{Conc1}
P\big\{|\sum_{j=0}^{n-1}\hat g_{j,k}(z,\zeta_j)-\sum_{j=0}^{n-1}\bbE\hat g_{j,k}(z,\zeta_j)|\geq c_1+t\big\}\leq c_2e^{-c_3\frac{t^2}n}.
\end{equation}
By taking $t=t_n$ which grows faster than $\sqrt n$ but slower linearly in $n$, we derive from Borel-Cantelli lemma that
\[
\lim_{n\to\infty}\big|\frac1n\sum_{j=0}^{n-1}\hat g_{j,k}(z,\zeta_j)-\frac 1n\sum_{j=0}^{n-1}\bbE\hat g_{j,k}(z,\zeta_j)\big|=0,\,\,P-\text{a.s.}
\] 
and so for any $\ve$, 
\[
\limsup_{n\to\infty}
\big|\frac 1n\sum_{j=0}^{n-1}\Pi_j(z)-\frac 1n\sum_{j=0}^{n-1} \bbE \Pi_j(z)\big|<\ve,\,\,P-\text{a.s.}
\]
which completes the proof of the theorem.
\end{proof}

Next, let $\cQ$ be a compact metric space and let $\bar S_j:\cX\to\cX$ and $\bar r_j:\cX\to\bbR$ satisfy the conditions specified in Section \ref{SDSsec}  with $\cE_j=\cQ$. Let $(\bar\la_j^{(0)},\bar h_j^{(0)},\bar\nu_j^{(0)})$ be the RPF triplets corresponding to the transfer operator generated by $\bar S_j$ and $\bar r_j$ and let $\bar\mu_j=\bar h_j\nu_j^{(0)}$ be the appropriate Gibbs measure. 
Let $\eta_j=\bar S_0^j\eta_0=\bar S_{j-1}\circ\cdots\circ\bar S_{1}\circ\bar S_{0}\eta_0$ be a sequence of random variables, where $\eta_0$ is distributed according to $\bar\mu_0$.

\begin{theorem}\label{II'}
Suppose that there exists  a probability measure $\textbf{n}$  on $\cY$ which assigns positive mass to open sets, so that $(S_j)_*\textbf{n}\ll\textbf{n}$ for any integer $j$ and  $r_j$ is defined by 
\[
e^{-r_j}=\frac{d(S_j)\textbf{n}}{d\textbf{n}}.
\]
Assume also that $S_j$,\,$j\geq0$ are random, and that they have form $S_j=S_{\eta_j}$.
Then there exists a function $ h:\cX\to\bbR$ so that, $\hat P$-almost surely, the distribution $P_j=\mu_j$ of all the $\zeta_j$'s is $\ka:= hd\textbf{n}$ (i.e. in the extension we have $\mu_j=\mu_{\xi_j}=\ka$ for each $j$).
\end{theorem}

\begin{proof}
Note first that 
the assumptions in the statement of the theorem mean that 
 $S_j$ and $r_j$ are chosen at random by the invertible extension $\hat\eta_j,\,j\in\bbZ$ of $\eta_j,\,j\geq0$ and that they depend only on the $0$-the coordinate.
We can assume that $S_j,\,j<0$ are chosen at random according to the $j$-the coordinate in this extension. In this case,
we only need to show that there exists a function $h:\cX\to\bbR$ so that for any $j\geq0$ we have 
$\textbf{h}^{(0)}_j= h$ ($\hat P$-a.s.).
 
Let $\beta_j\,j\geq0$ be independent copies of $\eta_0$, and consider the setup of i.i.d. maps $S_{\beta_j},\,j\geq0$. 
Suppose that the above i.i.d. process is defined on a probability space $(\Om,\cF,P)$ so that
$\beta_j=\te^j\beta_0$ for some $P$ preserving map $\te$. Consider the skew product map $S(\om,x)=(\te\om,S_{\beta_0(\om)} x)$.
Then, as in \cite{Aimino}, there exists an $S$-invariant measure of the form $P\times (hd\textbf{n})$ for some strictly positive continuous function $h:\cX\to\bbR$ so that $\textbf{n}(h)=1$. Exactly as in Section 4.1 in \cite{Annealed}, it follows that for $\bar\mu_0$-almost any $q\in\cQ$,
\[
\mathscr{L}_0^{(q)}\bar h=\bar h.
\]
Fix some $y\in\cQ$ and set 
\[
\Gam_y=\{q\in\cQ:\,\mathscr{L}_0^{(q)} h(y)= h(y)\}.
\]
Then $\bar\mu_0(\cQ\setminus\Gam_y)=0$ for each $y$.
Since $\bar\nu_j^{(0)}=\textbf{n}$ we have
\[
\bar\mu_j(\cQ\setminus\Gam_y)=\int_{\cQ\setminus\Gam_y}\frac{\bar h_j(t)}{\bar h_0(t)}d\bar{\mu}_0(t)=0
\] 
and therefore
\[
\bar{\mu}_0\big\{\bigcap_{j=0}^\infty (\bar S_0^j)^{-1}\Gam_y\big\}=1
\]
and therefore for any $y$ and $j\geq0$ we have 
\[
\mathscr{L}_0^{(\hat\eta_j)}h(y)=h(y),\,\,\text{a.s.}.
\]
Since both sides are continuous in $y$ and $\cY$ is compact we conclude that 
\[
\mathscr{L}_0^{(\hat \eta_j)}h=h,\,\hat P-\text{a.s.}
\]
for any $j\geq0$. Replacing $\hat\eta_0$ with 
$\hat\eta_k$ for any integer $k$, and using that $\hat \eta_{k+j}=\sig^j\hat\eta_k$ for any $j\geq0$ we derive that 
\[
\mathscr{L}_0^{(\hat\eta_k)}h=h
\]
for any $k$. Since $\bar\la_{j}(0)=1$ for any $j$, using (\ref{Exponential convergence}) we derive that
\[
\lim_{n\to\infty}\mathscr{L}_0^{(\hat\eta_{k-1})}\circ\mathscr{L}_0^{(\hat\eta_{k-2})}\circ\cdots\circ\mathscr{L}_0^{(\hat\eta_{k-n})}h=\textbf{n}(h)\textbf{h}_k=\textbf{h}_k
\]
and so $\textbf{h}_k=h$ for any $k$. Note that (\ref{Exponential convergence}) holds true for H\"older continuous functions, but for real $z$'s the converges itself, without rates, holds true for continuous functions by monotonicity arguments due to Walters (see the proof of Proposition 3.19 in \cite{MSU} and \cite{Walt}).
\end{proof}

\subsubsection{\textbf{Converges rate towards the moments}}\label{ConvR}

Suppose that $S_j,\,j\in\bbZ$ is a nonstationary subshift of finite type so that $\textbf{h}_j^{(0)}$ does not depend on $j$, i.e. that $(S_j)_*\mu=\mu$ for some probability measure $\mu$ (e.g. when $S_j$ is random).
Moreover, assume that $T_j,f_j$ and $u_j$ depend only on the $j$-the coordinate $X_j$. Then by  (\ref{Approx}) the random variables $\Pi_j(z)$ can be approximated exponentially fast in the $L^\infty$ norm by functions of the coordinates at places $j,j+1,...,j+n$. Taking into account Theorem  \ref{SFT pros} (ii), we conclude that all the conditions of Theorem 2.4 in \cite{HafMD} hold true with $\ell=1$, with $\Pi_j(z)$ in place of $\xi_j$ (from there) and with any bounded function $F$ which identifies with the function $G(x)=x$ on a compact set which contains all the possible values of all of the $\Pi_j(z)$'s. In particular, (2.11) from \cite{HafMD} holds true, and so for any $t\geq0$, $r\geq1$ and $n\geq1$,
\begin{equation}\label{Conc2}
P\big\{|\sum_{j=0}^{n-1}\Pi_j(z)-n\bbE\Pi_0(z)|\geq t+Cn\del^r+r\big\}\leq 2e^{-C\frac{t^2}{nr}}
\end{equation}
where $C$ is some positive constant. Taking $r$ of the form $r=a\ln n$ for an appropriate $a$ we derive that for some constant $C_1>0$, for any $t\geq0$ and $n\geq1$ we have
\[
P\big\{\big|\sum_{j=0}^{n-1}\Pi_j(z)-n\bbE\Pi_0(z)\big|\geq t+C_1\ln n\}\leq 2e^{-C_1\frac{t^2}{n\ln n}}.
\]
Taking $t$ of the form $t=t_n=a_1n^{\frac12}\ln n$ and using the Borel-Cantelli lemma we derive that $P$-a.s. for any sufficiently large $n$,
\[
\big|\frac1n\sum_{j=0}^{n-1}\Pi_j(z)-\bbE\Pi_0(z)\big|\leq (1+C_1)n^{-\frac12}\ln n.
\]
Since $\la_j(z)$ is analytic in $z$ and uniformly bounded in $z$ and $j$, it follows that for any $k\geq1$ there exists 
a constant $b_k$ so that for any sufficiently large $n$,
\[
\big|\frac1n\sum_{j=0}^{n-1}\Pi_j(z)-\bbE\Pi_0(z)\big|\leq b_kn^{-\frac12}\ln n
\]
and (\ref{MomRateTemp}) follows from Theorem \ref{MomThm.0}. The almost optimal Berry-Esseen inequality (\ref{BE temp}) follows by arguments similar to the ones in the proof of Theorem \ref{Berry-Esseen}.


\begin{thebibliography}{Bow75}

\bibliographystyle{alpha}
\itemsep=\smallskipamount






\bibitem{Aimino}
R. Aimino, M. Nicol and S. Vaienti. {\em Annealed and quenched limit theorems for random expanding dynamical systems}, Probab. Th. Rel. Fields 162, 233-274, (2015).

\bibitem{Arno}
P. Arnoux and A.Fisher, {\em Anosov families, renormalization and non-stationary
subshifts}, Erg. Th. Dyn. Syst., 25 (2005), 661-709





 
\bibitem{[3]} 
O. Butterley, and E. Peyman {\em Exponential mixing for skew products with discontinuities}, Trans. Amer. Math. Soc. 369 (2017), no. 2, 783-803.
 





\bibitem{castro}
A. Castro, P. Varandas. {\em Equilibrium states for non-uniformly expanding maps: decay of
correlations and strong stability}. Annales de l'Institut Henri Poincar\'e - Analyse non Lineaire,
 (2013) 225-249,

\bibitem{Coelho}
Z. Coelho, W. Parry {\em Central limit asymptotics for shifts of finite type}, Israel J. Math.
69, (1990), no. 2, 235-249.

\bibitem{Conze}
J-P Conze and A. Raugi, {\em Limit theorems for sequential expanding dynamical
systems}, AMS 2007.



\bibitem{CV}
C. Castaing and M.Valadier, {\em Convex analysis and measurable multifunctions}, 
Lecture Notes Math., vol. 580, Springer, New York, 1977.

 \bibitem{Bow}
R. Bowen, {\em Equilibrium states and the ergodic theory of Anosov diffeomorphisms}, second revised edition, Lecture Notes in Mathematics, Springer Verlag, 2008. 


\bibitem{Demb}
 A. Dembo and O. Zeitouni, {\em Large Deviations Techniques and Applications}, 2nd edn. Applications of Mathematics, vol. 38. Springer, New York (1998).

\bibitem{Dor}
H. D\"oring and P. Eichelsbacher, {\em Moderate deviations via cumulants}, J. Theor. Probab. 26 (2013), 360-385.

\bibitem{Douk}
P. Doukhan and M.Neumann, {\em Probability and moment inequalities for sums of weakly dependent random variables, with applications}, Stochastic Process. Appl. 117 (2007), 878-903.


\bibitem{drag}
D. Dragi\v{c}evi\'c, G. Froyland, C. Gonz\'alez-Tokman and S. Vaienti,
{\em A spectral approach for quenched limit theorems for random expanding dynamical systems},
Commun. Math. Phys. 360, 1121-1187 (2018).

\bibitem{drag1}
D. Dragi\v{c}evi\'c, G. Froyland, C. Gonz\'alez-Tokman and S. Vaienti,
{\em A spectral approach for quenched limit theorems for random hyperbolic dynamical systems},
To appear in  Trans. Amer Math. Soc.



\bibitem{Dub1}
L. Dubois, {\em Projective metrics and contraction principles for
complex cones}, J. London Math. Soc. 79 (2009), 719-737.

\bibitem{Dub2}
L. Dubois, {\em An explicit Berry-Ess\'een bound for uniformly expanding maps on
 the interval}, Israel J. Math. 186 (2011), 221-250.



\bibitem{GH}
Y. Guivar\'ch and J. Hardy, {\em Th\'eor\`emes limites pour une classe de cha\^ines de Markov et
applications aux diff\'eomorphismes d'Anosov}, Ann. Inst. H. Poincar\'e Probab. Statist. 24 (1988), no. 1, 73-98.


\bibitem{book}
Y. Hafouta and Yu. Kifer, {\em Nonconventional limit theorems and random dynamics}, 
World Scientific, Singapore, 2018.






\bibitem{HafMD}
Y. Hafouta, {\em Nonconventional moderate deviations theorems and exponential concentration inequalities}, to appear in  Ann. Inst. H. Poincar\'e Probab. Statist.


\bibitem{Annealed} Y. Hafouta, {\em Limit theorems for some skew products with mixing base maps}, to appear in Ergodic Theory Dynam. Systems, DOI: https://doi.org/10.1017/etds.2019.48.



\bibitem{HafEdge}
Y. Hafouta, {\em Asymptotic moments and Edgeworth expansions for some processes in random dynamical environment}, arXiv preprint 1812.06924.


\bibitem{FunCLT}
Y. Hafouta {\em A functional CLT for nonconventional polynomial arrays}, preprint: arXiv:1907.03303.


\bibitem{ASIP me}
Y. Hafouta {\em Almost sure invariance principle for time dependent non-uniformly expanding dynamical systems}, preprint arXiv 1910.12792.


 
\bibitem{HH}
H. Hennion and L. Herv\'e, {\em Limit Theorems for Markov Chains and Stochastic Properties of Dynamical
Systems by Quasi-Compactness}, Lecture Notes in Mathematics vol. 1766, Springer, Berlin, 2001.


\bibitem{Hyd2}
Nicolai Haydn, Matthew Nicol, Andrew Török and Sandro Vaienti, {\em Almost sure invariance principle for sequential and non-stationary dynamical systems}, Trans. Amer. Math. Soc. 369 (2017), 5293-5316.
  
\bibitem{IL}
I.A. Ibragimov and Y.V. Linnik, {\em Independent and Stationary Sequences of Random Variables}, Wolters-Noordhoff, Groningen, 1971








\bibitem{Kifer-1996}
Yu. Kifer, {\em Perron-Frobenius theorem, large deviations, and random perturbations
in random environments},
 Math. Z. 222(4) (1996), 677-698.

\bibitem{Kifer-1998}
Yu. Kifer, {\em Limit theorems for random transformations and processes in random
environments},
Trans. Amer. Math. Soc. 350 (1998), 1481-1518.

\bibitem{Kifer-Thermo}
 Yu. Kifer, {\em Thermodynamic formalism for random transformations revisited},
Stoch. Dyn. 8 (2008), 77-102.


\bibitem{Zemer}
A Korepanov, Z Kosloff, I Melbourne
{\em Martingale-coboundary decomposition for families of dynamical systems}
 Annales de l’Institut Henri Poincar\'e C, Analyse non lin\'eaire 35, no. 4 (2018) 859-885.


\bibitem{Liverani}
K. Fernando and C. Liverani, {\em Edgeworth expansions for weakly dependent random variables}, preprint, arXiv 1803.07667, 2018.

\bibitem{Mil}
V.D. Milman and G. Schechtman, {\em Asymptotic theory of finite-dimensional normed spaces}, Lecture
Notes in Mathematics, Vol. 1200, Berlin: Springer-Verlag, 1986, (With an appendix by M.
Gromov).

\bibitem{MSU}
V. Mayer, B. Skorulski and M. Urba\'nski, {\em Distance expanding random mappings,
thermodynamical formalism, Gibbs measures and fractal geometry},
Lecture Notes in Mathematics, vol. 2036 (2011), Springer.
 

\bibitem{Neg1}
S.V. Nagaev, {\em Some limit theorems for stationary Markov chains},
Theory Probab. Appl. 2 (1957), 378-406.

\bibitem{Neg2}
S.V. Nagaev, {\em More exact statements of limit theorems for homogeneous Markov chains}, 
Theory Probab. Appl. 6 (1961), 62-81.


\bibitem{Nicol}
M. Nicol, A. Torok, S. Vaienti, {\em Central limit theorems for sequential and random intermittent dynamical systems}, 
Ergodic Theory Dynam. Systems, 38, pp. 1127-1153, 2016.
 

\bibitem{Rug}
H.H. Rugh, {\em Cones and gauges in complex spaces: Spectral gaps and complex
 Perron-Frobenius theory}, 
Ann. Math. 171 (2010), 1707-1752.


\bibitem{Shir}
A.N. Shiryaev, {\em Probability}, Springer-Verlag, Berlin, 1995.

\bibitem{Stat}
L. Saulis and V.A. Statulevicius, {\em Limit Theorems for Large Deviations}, Kluwer Academic, Dordrecht, Boston, 1991.




\bibitem{Vara}
P. Varandas, M. Viana, {\em Existence, uniqueness and stability of equilibrium states for non-uniformly expanding maps}, Ann. Inst. H. Poincare Anal. Non Lineaire 27 (2010) 555-593.


\bibitem{Walt}
P. Walters, {\em Invariant measures and equilibrium states for some mappings which expand distances}.
Trans. Amer. Math. Soc. 236, 121-153 (1978).


\bibitem{Young1}
L.S. Young, {\em Statistical properties of dynamical systems with some hyperbolicity}, 
 Ann. Math. 7 (1998) 585-650.

\bibitem{Young2}
L.S. Young, {\em  Recurrence time and rate of mixing}, Israel J. Math. 110 (1999) 153-88.

\bibitem{Young3}
L. Rey-Bellet and L.S. Young. {\em Large deviations in non-uniformly hyperbolic dynamical systems}, Ergodic Theory Dynam. Systems, 28, 587–612, 2008.












\end{thebibliography}
\end{document}